\documentclass[12pt]{amsart}
\usepackage[utf8]{inputenc}
\usepackage{color}
\usepackage{times}
\usepackage{subcaption}

\usepackage{amsmath,amsthm,amsfonts,amssymb}
\usepackage{graphicx}
\usepackage{dsfont}

\makeatletter
\def\namedlabel#1#2{\begingroup
 #2%
 \def\@currentlabel{#2}%
 \phantomsection\label{#1}\endgroup
}
\makeatother

\theoremstyle{plain}
\newtheorem*{theorem*}{Theorem}
\newtheorem*{thmex*}{Theorem~\ref{example}}
\newtheorem*{thmasymp*}{Theorem~\ref{thmAsymp}}
\newtheorem{theorem}{Theorem}[section]

\newtheorem{corollary}[theorem]{Corollary}

\newtheorem{lemma}[theorem]{Lemma}

\newtheorem{example}[theorem]{Example}

\theoremstyle{definition}
\newtheorem{definition}[theorem]{Definition}

\newcommand{\ben}{\begin{enumerate}}
\newcommand{\een}{\end{enumerate}}

\newcommand{\ed}{\end{document}}

\definecolor{rrr}{rgb}{.9,0,.1}

\definecolor{rr}{rgb}{.8,0,.3}

\newcommand{\tr}{\triangleright}

\usepackage{pdfsync}
\usepackage[utf8]{inputenc}
\graphicspath{ {./images/} }

\title[Knot theory for proteins]{Knot theory for proteins: Gauss codes, quandles and bondles}

\author[C. Adams, J. Devadoss]{Colin Adams, Judah Devadoss}
\address{Williams College, Williamstown, Massachusetts, USA}
\email{cadams@williams.edu,judahdev@gmail.edu}
\author[M. elhamadi]{Mohamed Elhamdadi}
\address{University of South Florida, Tampa, Florida, USA}
\email{emohamed@mail.usf.edu}
\author[A. Mashaghi]{Alireza Mashaghi}
\address{Leiden University, Leiden, The Netherlands}
\email{a.mashaghi.tabari@lacdr.leidenuniv.nl}

\begin{document}

\begin{abstract}

Proteins are linear molecular chains that often fold to function. The topology of folding is widely believed to define its properties and function, and knot theory has been applied to study protein structure and its implications. More that 97\% of proteins are, however, classified as unknots when intra-chain interactions are ignored. This raises the question as to whether knot theory can be extended to include intra-chain interactions and thus be able to categorize topology of the proteins that are otherwise classified as unknotted. Here, we develop knot theory for folded linear molecular chains and apply it to proteins. 

For this purpose, proteins will be thought of as an embedding of a linear segment into three dimensions, with additional structure coming from self-bonding. We then project to a two-dimensional diagram and consider the basic rules of equivalence between two diagrams. We further consider the representation of projections of proteins using Gauss codes, or strings of numbers and letters, and how we can equate these codes with changes allowed in the diagrams. Finally, we explore the possibility of applying the algebraic structure of quandles to distinguish the topologies of proteins. Because of the presence of bonds, we extend the theory to define bondles, a type of quandle particularly adapted to distinguishing the topological types of proteins.
\end{abstract}

\maketitle

\section{Introduction}
Folded linear molecular chains are ubiquitous in biology. Proteins and nucleic acids are linear polymers responsible for most cellular functions, for the inheritance of biological information, and are subject to changes during evolution and pathologies \cite{Dobson2002, Corces2018}. These chains often fold to function and their 3D structure contains information about their dynamics, evolution, and inter-molecular interactions and can be used for designing drugs \cite{Kantidze2019, Amaral2017}. Geometric and chemical properties of folded proteins and genomic DNA have been widely studied using various methods including NMR spectroscopy, X-ray crystallography, chromosome conformation capture, and mass spectrometry among others. Topological properties of these molecules have remained relatively unexplored due to lack of a relevant conceptual framework. Knot theory was successfully applied to study proteins and nucleic acids and protein and DNA knots were studied using various experimental techniques including nanopore technology and probe microscopy \cite{Taylor2003, Sulkowska2008, Soler2013, Mallam2010}.  Despite being interesting and innovative, these studies have had a limited impact on protein science as the vast majority of identified proteins fall into one topology class, i.e., the unknot \cite{Sulkowska2012}. 

Thus, standard knot theory cannot be effectively used to classify proteins \cite{mashaghi2014circuit}. Another shortcoming of the standard knot theoretic approach is that intra-molecular interactions or contacts are ignored. These interactions drive the folding of the molecular chains and are functionally important \cite{mashaghi2015circuit, C4CP03402C, heidari2017topology, satarifard2017topology, mashaghi2019endrestrained}. When intra-chain interaction is taken into consideration, the prevalence of knots and links substantially increases \cite{7, Dabrowski17, Dabrowski19}.

Thus, there is a need for a new topology framework that includes intra-chain interactions and is able to classify fold topology of biomolecular chains and in particular the proteins.  This paper presents a new knot theory for folded linear molecular chains and looks to classify the topological structure of proteins through the application of  certain aspects of knot theory. More specifically, we will apply a modified singular knot theory, Gauss codes to keep track of the structure, and an associated singular quandle called a bondle to distinguish structures.

Proteins are continuous linear molecules with the ends unbonded, so we will look at them as linear segments embedded in three-space. Protein structure for a single protein is formed on three different levels, shown in Figure \ref{fig:11}. The Primary Structure is defined by the ordering of the amino acids, or building blocks of the protein. These amino acids are bonded in a sequential chain from the N-terminus to the C-terminus, with each amino acid presenting an exposed R-group that can interact chemically or through the electrostatic effect with other R-groups and other molecules. 
The Secondary Structure is defined by the coiling or local structuring of the amino acids. This is where structural patterns appear such as $\beta$-pleated sheets and $\alpha$-helices, both held together by hydrogen bonds. 
Tertiary Structure is defined by the interactions between different R-groups or backbone interactions, forming a structure for the entire protein.

\begin{figure}[htbp]
  \includegraphics[width=.95\linewidth]{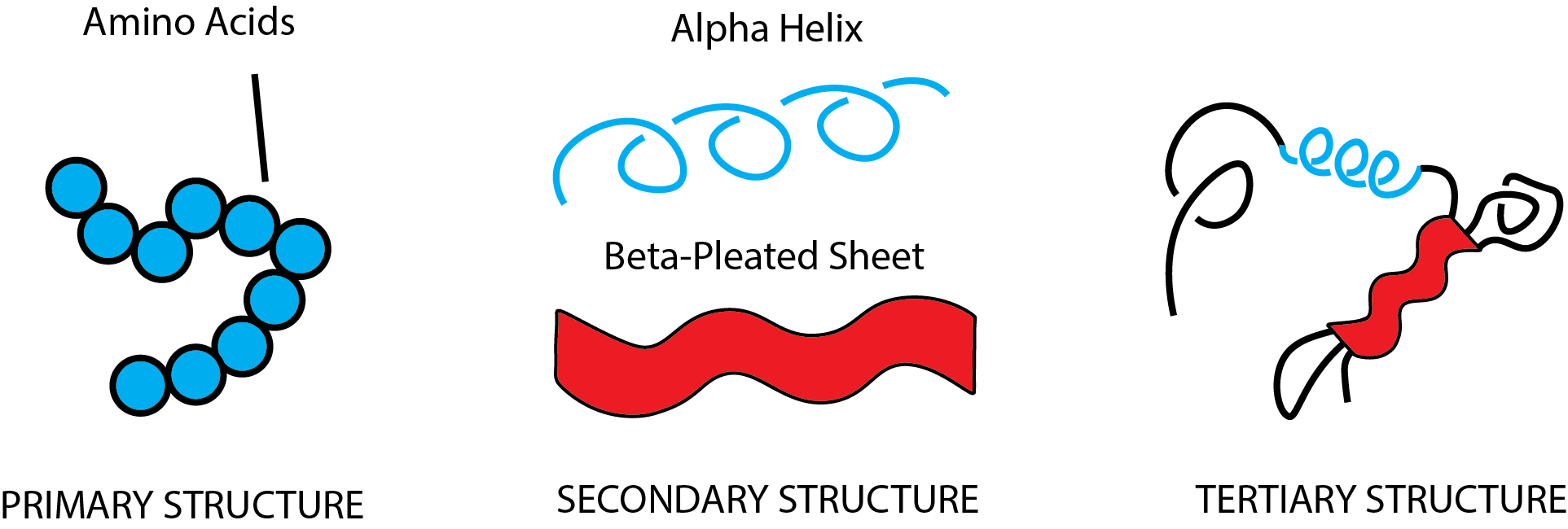}
  \caption{Depiction of protein structures.}
  \label{fig:11}
\end{figure}

A \textit{knot} is an embedding of S$^1$, the circle, into three dimensions, called a {\it conformation}, given by a function $f:[0 , 2\pi] \rightarrow \mathds{R}^3$ where $f(0)=f(2\pi)$. We define an equivalence on conformations of knots, considering two conformations equivalent if there is an ambient isotopy from one to the other. This means that we can deform the one through space to the other without passing the knot through itself. 
As is common, we will use the word ``knot'' to represent both a given conformation of a knot and the equivalence class of conformations corresponding to the given conformation, clarifying which is which when necessary.

Knot theory can be extended to singular knot theory, where we allow a finite number of singular points where two points on the circle are sent to the same point in 3-space by $f$. (See \cite{Dan18} for more on singular knots.) Allowing singularities will provide us with the ability to model intra-molecular bonds.

While it is often preferable to describe a knot in three dimensions, it is not always tractable. Therefore, we project the knot in a particular direction onto a plane, obtaining a   \textit{projection}. We consider only regular projections where a finite number of pairs of points on the knot are identified with each other and result in what we call crossings. We keep track of which of the two points in the pair is the top one. A projection is essentially a shadow that retains information at the crossings of the strands. Using a defined set of rules called Reidemeister moves, we can change one projection of a knot to any other projection of that same knot. In a two-dimensional projection, we define a \textit{classical crossing} as a place where one strand goes over another, and if we are allowing singularities, a \textit{singular crossing}  where two or more strands intersect each other at a point in the three-dimensional conformation.

We can similarly think of proteins as a conformation of a line segment [0,1] in three dimensions, which we will call a {\it protein model}, with an associated function $f:[0 , 1] \rightarrow \mathds{R}^3$. If we utilize the same equivalence of ambient isotopy, then every protein can be disentangled and they are all equivalent. But once we include singularities reflecting bonding of points along the conformation, this will no longer be true.

In a protein, we define singular crossings to exist where a protein has intra-chain interactions (also called contacts). These contacts take one of two forms. The first is covalent bonds, defined as two atoms sharing electrons. Of special interest to proteins are disulfide bridges, in which two thiol R-groups, made of one sulfur and one hydrogen, bond and release their hydrogen atoms. The second type of bond is formed through non-covalent interactions. A major form of non-covalent interactions are the electrostatic ones, particularly in the form of hydrogen bonds, which are partially electrostatic, but often come in multiples, making their strength significant enough to control the protein's structure. Hydrogen bonds mediate the interactions between beta strands and the formation of alpha-helical structures. 


Proteins differ from knots in that they have  two endpoints that are not connected. One approach to this difference has been work on modelling proteins as \textit{knotoids}, which are the projection of the conformations of a linear segment \cite{7, GGLDSK}. Reidemeister moves (which we will look at in the next section) apply in the theory of knotoids, but in this theory, the endpoints cannot be moved across other strands in the projection. Although this may be useful when dealing with proteins that are somewhat rigid and have so-called lassos, here we desire a general theory that allows the movement of endpoints across strands in a projection.  Thus we do not consider proteins as they relate to knotoids. 

In finding a notation for protein structure, it is true that proteins have variable flexibility and restricted length, providing more limitations than the ones we place on curves in 3-space when defining knots. But the goal of this paper is topological in nature. Thus, we do not capture the full sense of rigidity or steric hindrance in a protein. Since we are allowing for the deformation of a protein's strands in the following sections, we will at times allow the same for the ends of the strand. The notation defined in this paper looks to balance simplicity and mathematical utility with chemical precision. 

In Section 2, we discuss Reidemeister moves, which are moves one can do on a projection of a knot to obtain a new diagram of the same knot. We extend them to allow features present in proteins, including bonds, endpoint $\beta$-pleated sheets and $\alpha$-helices.

In Section 3, we introduce Gauss codes, which can be used to describe in symbols a projection of a knot. We extend them to proteins.

In Section 4, we introduce quandles, which are algebraic objects that can be used to distinguish between different knots. In Section 5, we introduce singquandles, which are an extension of quandles that have been used to distinguish knots with singularities. We further extend this idea to the idea of a bondle, which is a quandle that can be applied in the presence of  bonds.

In Section 6, we introduce the oriented bondle, which seems particularly suited to distinguishing between the topological types of proteins. We then identify several families of oriented bondles. In Section 7, we provide several examples of pairs of proteins that can be distinguished using oriented bondles.

\section{Reidemeister Moves}

\textit{Reidemeister moves} are a set of changes to the combinatorial pattern that is a projection of a knot. Planar isotopy is a deformation of the projection that does not change the combinatorial pattern.  The critical result from \cite{AB26} or \cite{Rei27}  says that two knots $K_1$ and $K_2$ are equivalent if and only if there exists a sequence of Reidemeister moves and planar isotopy in the plane that transforms a projection of $K_1$ to a projection of $K_2$. The three Reidemeister moves are depicted in  Figure \ref{fig:21}. We refer to the monogonal face found in a Type I Reidemeister move as a {\it kink}. 

\begin{figure}[htbp]
  \includegraphics[width=.95\linewidth]{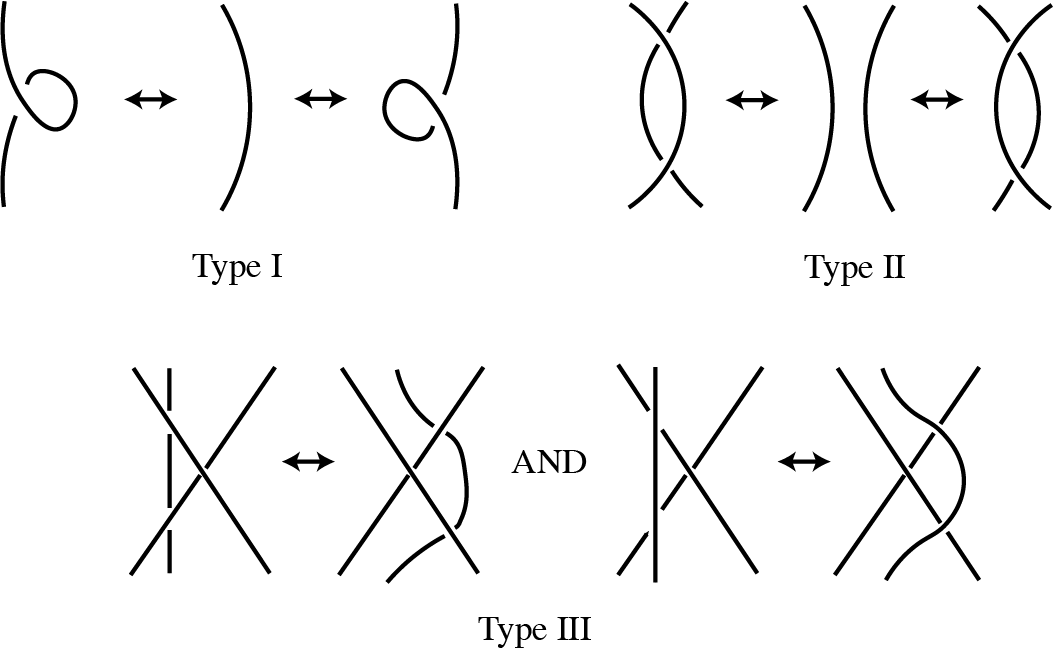}
  \caption{Types I, II, and III Reidemeister Moves}
  \label{fig:21}
\end{figure}

While these three Reidemeister moves are able to completely describe equivalence in classical knots, for our purposes we need to add in  singular crossings to represent the covalent and hydrogen bonds that proteins form with themselves (\cite{Dan18}). The use of singular crossings requires an additional two Reidemeister moves as shown in Figure \ref{fig:22} (\cite{Yua17, AE}).

\begin{figure}[htbp]
  \includegraphics[width=.95\linewidth]{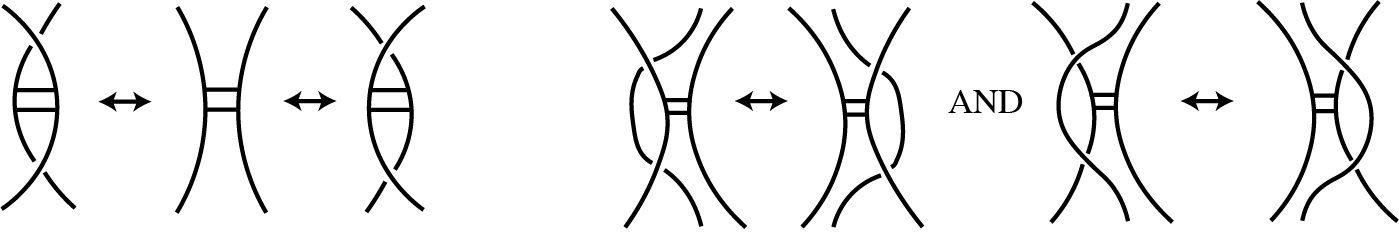}
  \caption{Type IV and V Reidemeister Moves for Singularities.}
  \label{fig:22}
\end{figure}

As appearing in this illustration, we denote a singularity by a small rectangle with two parallel edges on two strands. By drawing singular crossings in this fashion, we show that the strands do not cross over each other, but are instead bound together to more closely replicate the structure we see in proteins. By doing so, we do not lose any of the structure or properties we would hope to retain. 

In a protein, a $\beta$-pleated sheet consists of multiple segments of a protein that run parallel to each other, roughly in a plane, with hydrogen bonds connecting each segment to its adjacent segment in multiple places. If we collapse the sheet to a point, this functionally looks like a singular crossing with more than two strands. 
Therefore, in order to represent $\beta$-pleated sheets, we must extend singular knot theory to contain multi-singularities. We define  a \textit{multi-singularity} as a place where two or more strands intersect each other at a single point, as shown in Figure \ref{fig:23}. 

\begin{figure}[htbp]
  \includegraphics[width=.6\linewidth]{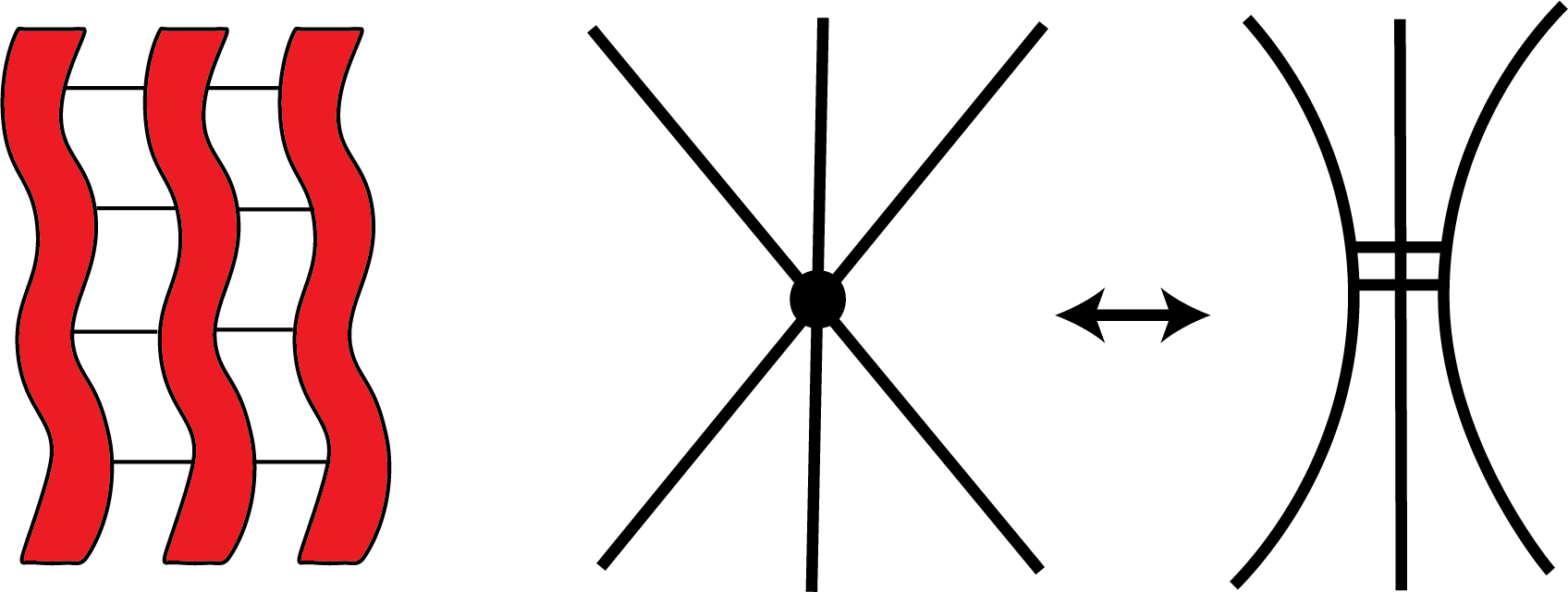}
  \caption{$\beta$-pleated sheets and multi-singular crossings.}
  \label{fig:23}
\end{figure}

We can then incorporate this into the moves shown in Figure \ref{fig:22} and define the set of singular Reidemeister moves shown in Figure \ref{fig:24} to describe topological isotopy. 

\begin{figure}[htbp]
  \includegraphics[width=.9\linewidth]{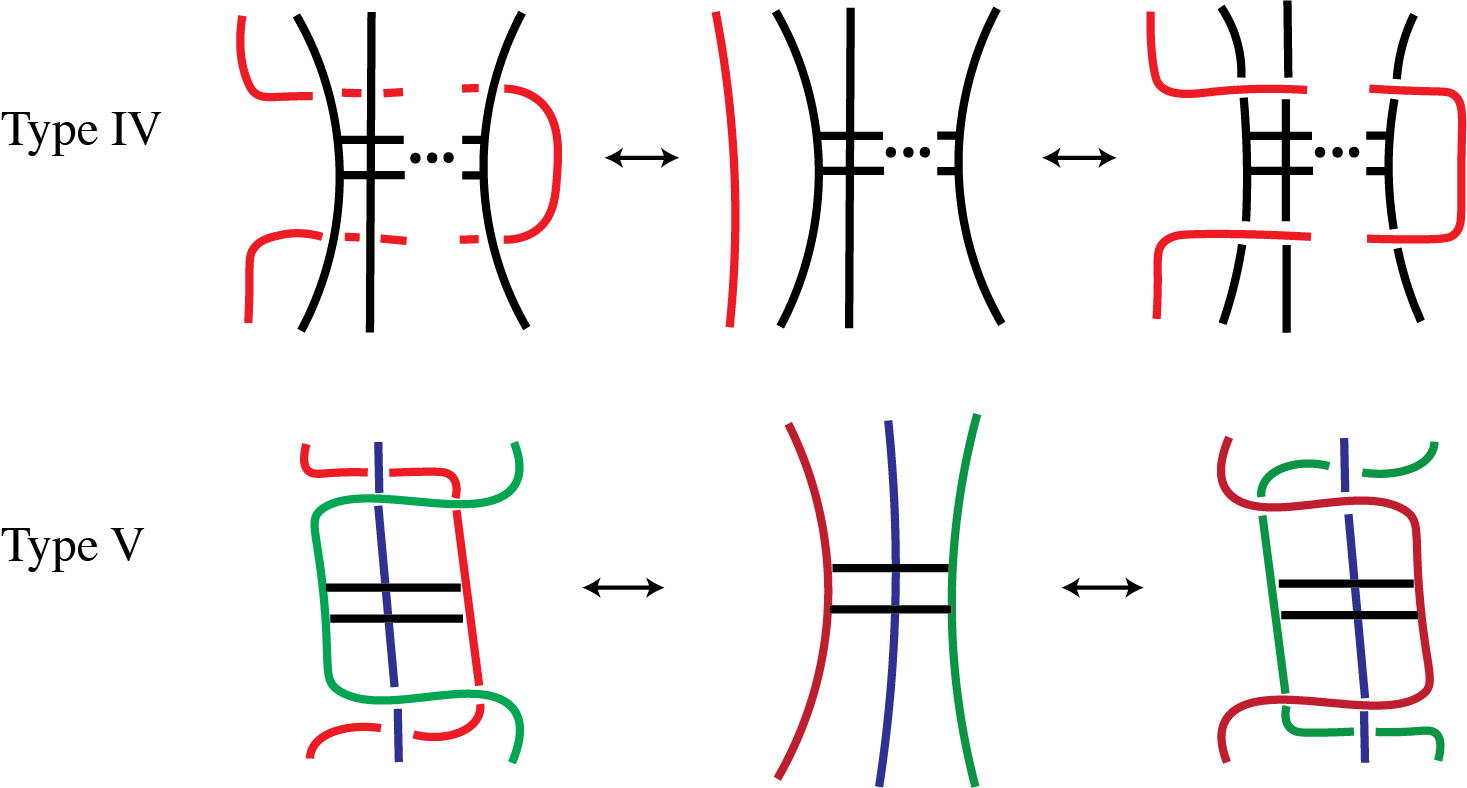}
  \caption{Types IV and V Reidemeister moves for multi-singular crossings.}
  \label{fig:24}
\end{figure}

Note that when applying a Type V Reidemeister move to a multi-singular crossing, we allow any number of segments to be included, even though we have only depicted three segments to simplify the illustration. When the singularity is flipped, this causes a half-twist of all of the segments above and below the multi-singular crossing in the process. 

Whereas in a Type IV move we only explicitly allow a segment to pass through a multi-singularity from left to right, we can use the given Reidemeister moves to show that we can slide a horizontal segment past a multi-singularity from top to bottom a well, as shown in Figure \ref{fig:25}.

\begin{figure}[htbp]
  \includegraphics[width=.95\linewidth]{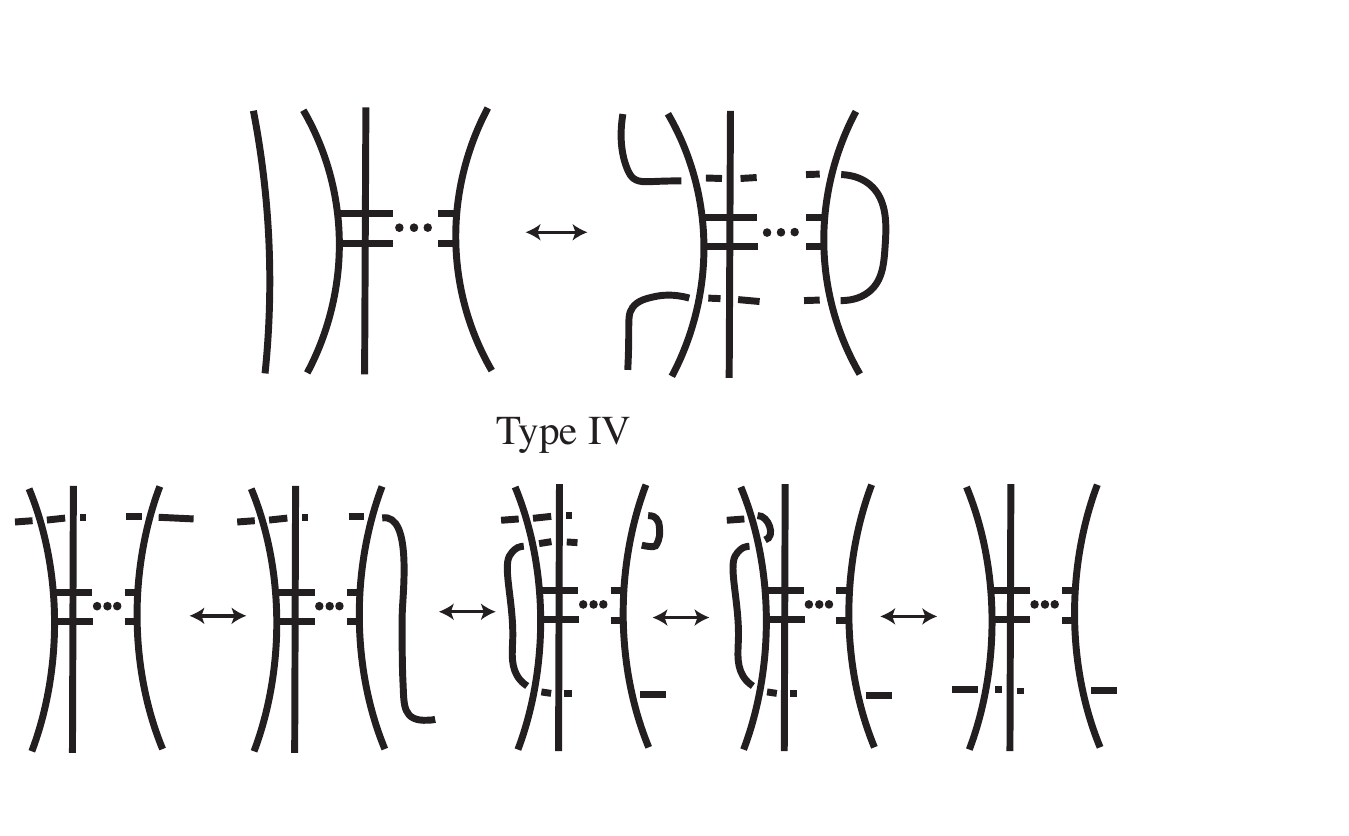}
  \caption{A vertical Type IV Reidemeister move produces a horizontal Type IV Reidemeister move.}
  \label{fig:25}
\end{figure}

Proteins also contain $\alpha$-helices, which we can represent in multiple ways. 
An $\alpha$-helix appears where a segment of a protein coils, with hydrogen bonds holding the coils together, as shown in Figure \ref{fig:11}. For now, we simply say that they do not add topological structure, but are important in protein function. In a projection, we allow $\alpha$-helices to slide along a segment, freely traversing a classical crossing, but not to pass through a singular crossing. Thus, we create a Reidemeister move Type VI to allow $\alpha$-helices, depicted as the small set of jagged lines, to pass over or under other segments as in Figure \ref{fig:26}.

\begin{figure}[htbp]
  \includegraphics[width=.95\linewidth]{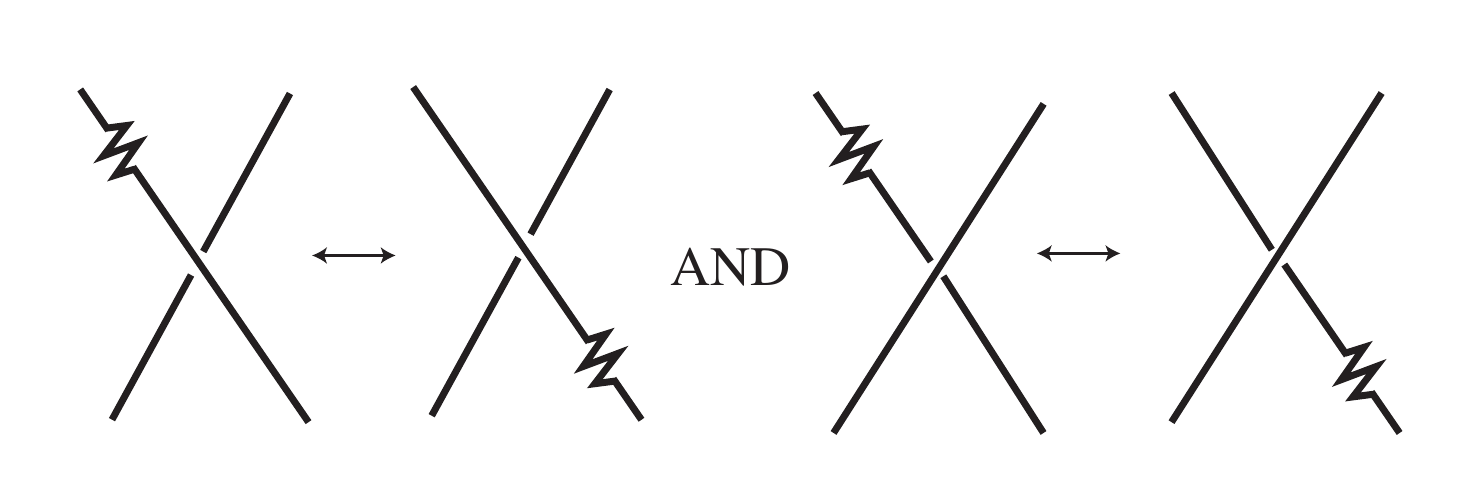}
  \caption{Type VI Reidemeister Move.}
  \label{fig:26}
\end{figure}

Since a protein is a continuous strand, we are able to equate its projection to a segment of a knot, allowing us the aforementioned Reidemeister moves. But because the protein has two unbonded endpoints, denoted $N$ and $C$,  that are free to move around, we need to be able to slide these endpoints past another strand.  Therefore, we define a final Reidemeister Type VII move for this action as well as in Figure \ref{fig:27}. 

\begin{figure}[htbp]
  \includegraphics[width=.95\linewidth]{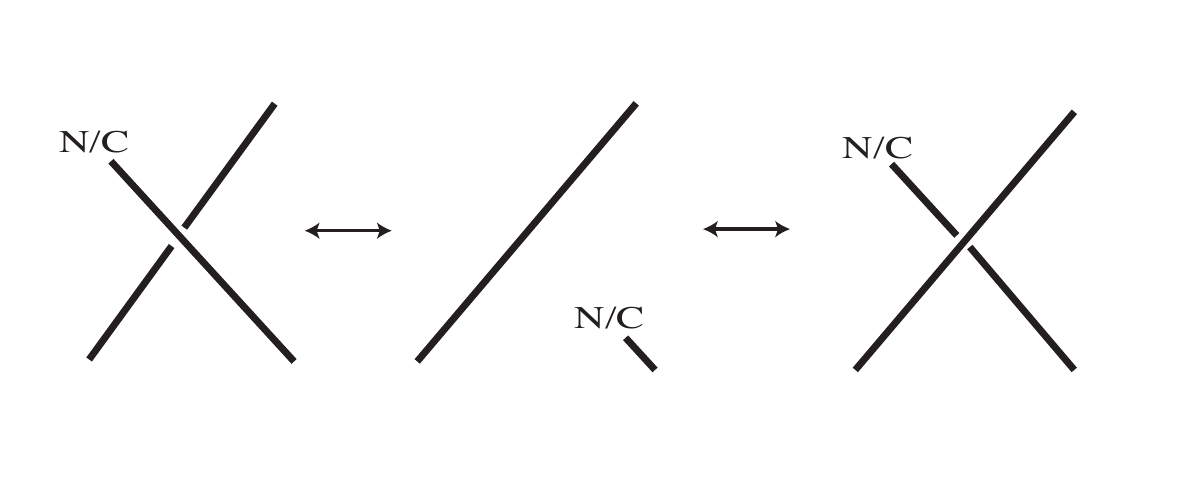}
  \caption{Type VII Reidemeister Move.}
  \label{fig:27}
\end{figure}


We would like to prove that this set of seven moves captures all possible equivalences between projections of protein conformations. To  achieve this, as is done in the knot case as well, we replace smooth conformations with piece-wise linear conformations, meaning that our conformation can be represented by a finite number of line segments glued end-to-end. 

\begin{theorem}Given two protein conformations, they are equivalent if and only if there is a sequence of planar isotopies and Reidemeister moves from this set of seven moves to get from a projection of the first to a projection of the second.
\end{theorem}

\begin{proof}As previously mentioned, it has already been proved  that for a classical knot, planar isotopy and  Type I, Type II, and Type III moves suffice to represent ambient isotopy in a knot (see \cite{Kau89} for a readable proof). The proof uses {\it triangle moves} to realize ambient isotopy between two piecewise linear knots. A triangle move is realized by taking a solid triangle that intersects the knot in one or two edges and  replacing those line segments on the knot by the non-intersecting edges on the triangle, as shown in Figure \ref{fig:28}.  Any isotopy we attain from deforming the original conformation can be represented by triangle moves. When we project down to a projection, one can show the triangle moves appear as planar isotopy and Reidemeister moves.  

\begin{figure}[htbp]
  \includegraphics[width=.95\linewidth]{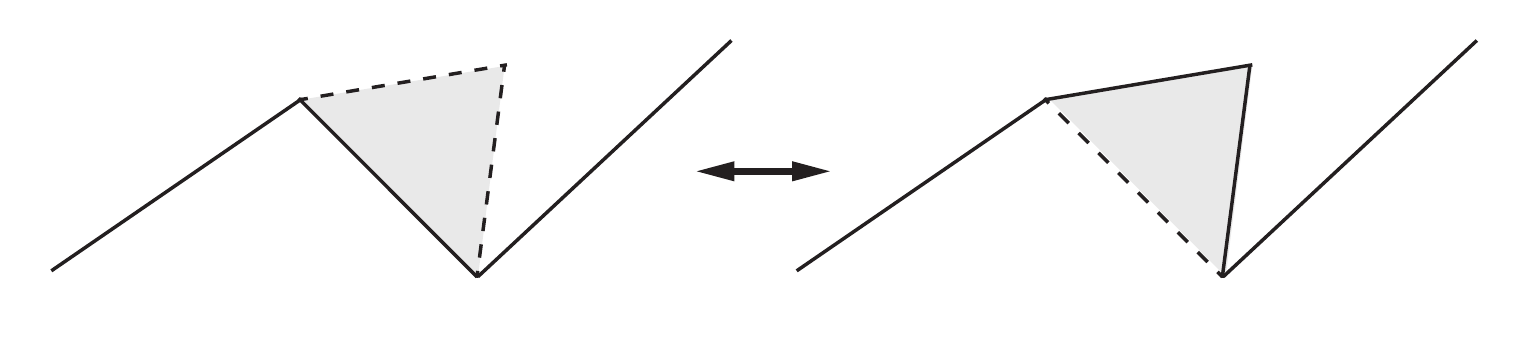}
  \caption{Example of Triangle Move}
  \label{fig:28}
\end{figure}

 A spatial graph is a conformation of a graph (consisting of a collection of edges sharing a collection of vertices as their endpoints) in 3-space. A rigid vertex graph further posits that adjacent edges coming into a vertex cannot twist about one another. The vertex can be flipped, intertwining the edges appropriately as in our Type V move. In \cite{Kau89}, it is  shown that two conformations of a rigid vertex spatial graph with all vertices of valency 4 (called an RV4 graph) are equivalent if and only if there is a sequence of planar isotopies and Reidemesister moves of Types I, II, III, IV, and V from a projection of one to a projection of the other. This proof also considers how triangle moves impact the projection. 
 
 Our situation for a protein conformation has three differences from this one. First, we allow our vertices to have an even number of edges that can be four or greater. However, this does not impact the proof as in \cite{Kau89} and it goes through exactly as before. 
 
 Second, we have $\alpha$-helices. These can be treated as vertices of valency two, and then the same arguments as in \cite{Kau89} go through to generate the Type VI Reidemeister move. 
 
 Third, we have the N and C endpoints of the protein. But it is straightforward to see that a triangle move that projects to overlap an endpoint will simply generate the Type VII Reidemeister move.

Thus, all isotopies allowed in deforming our structure can be represented by triangle moves, which in the projections can be represented by the defined Reidemeister moves. Hence, the seven moves are sufficient to convert one projection of a protein conformation to any projection of a topologically equivalent protein conformation. 

\end{proof}

With these seven Reidemeister moves, we are able to transform one projection of a protein to another.  However, it is inconvenient to convey these transformations through diagrams.  Encoding diagrams as strings of text allows us to both interpret and transmit the information. Therefore, we turn to Gauss codes to represent protein structure in text form.

\section{Gauss Codes}

Gauss codes are a means to represent knot projections using only symbols instead of diagrams. Using a Gauss code, we are able to easily go back and forth between the code and a projection while also being able to change entries in the code when Reidemeister moves are applied.

Since proteins are constructed and written from the N-terminus to the C-terminus, there is a definitive start and end point to proteins. Therefore, we start the Gauss code of a protein projection with $N$ and end the code with $C$. We also have a natural orientation to the protein. Crossings are oriented as in Figure \ref{fig:31}.

\begin{figure}[htbp]
  \includegraphics[width=.5\linewidth]{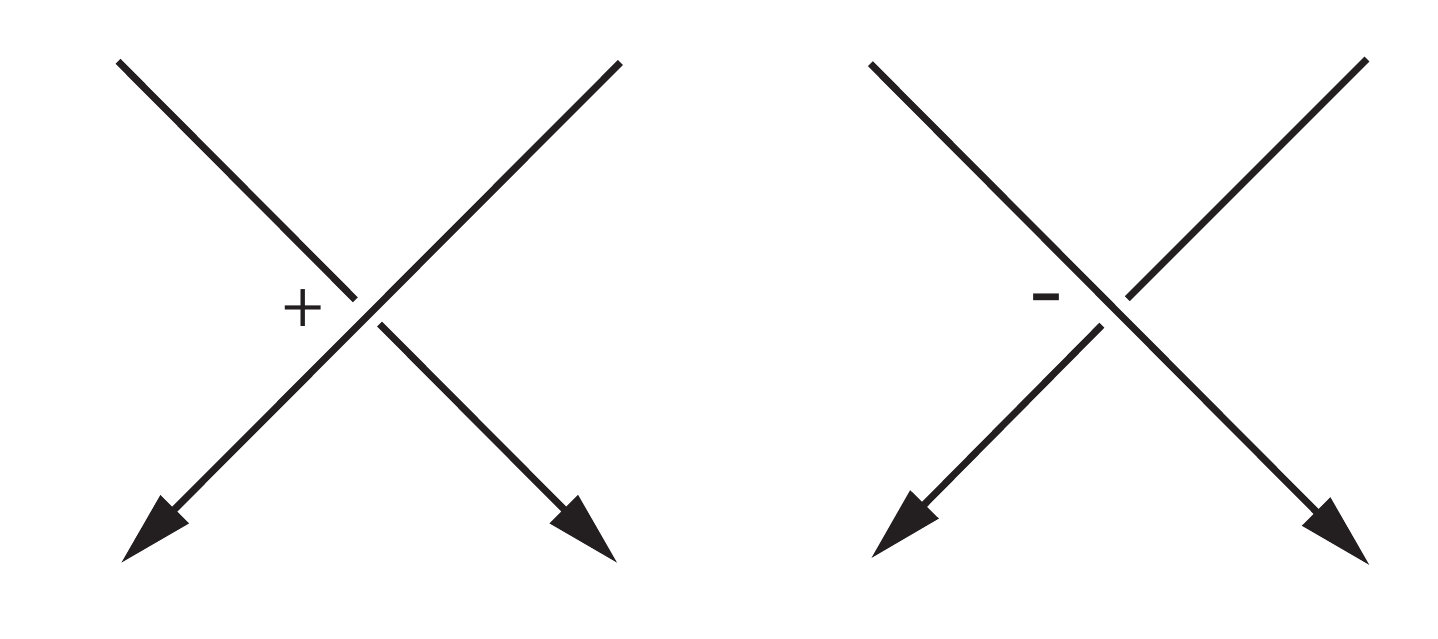}
  \caption{Positive and negative oriented crossings.}
  \label{fig:31}
\end{figure}

As with traditional Gauss codes, when following a strand, if the strand crosses over another strand, we denote this with an $O$ for over. Similarly, when traversing beneath another strand, we denote this with a $U$ for under. We assign an orientation to the crossing denoted by a superscript of $\pm$. 

If an $\alpha$-helix appears, we denote it as $\alpha$ with a superscript of $+$ if the helix is right-handed (coils clockwise) and $-$ if the helix is left-handed (coils counterclockwise). 
Bonds are written as $B$, and $\beta$-pleated sheets are written as $\beta$. Strands in bonds and $\beta$-pleated sheets are labeled with a superscript of $\pm$, using $+$ if the strand runs parallel to the strand that first occurs in the bond or sheet (so the first strand always receives a +), and $-$ if the strand runs anti-parallel to the strand that first occurs.

With $\beta$-pleated sheets, strands are numbered with a subscript. The zero strand is defined as the first strand that appears in the protein's sequence. Numbers are assigned as sequential integers to the left and right of the initial strand, with positive integers appearing on the side to which the second strand appears in the sequence, and all strands on the opposite side of the initial strand are defined as negative, as in Figure \ref{fig:32}. 

\begin{figure}[htbp]
  \includegraphics[width=.4\linewidth]{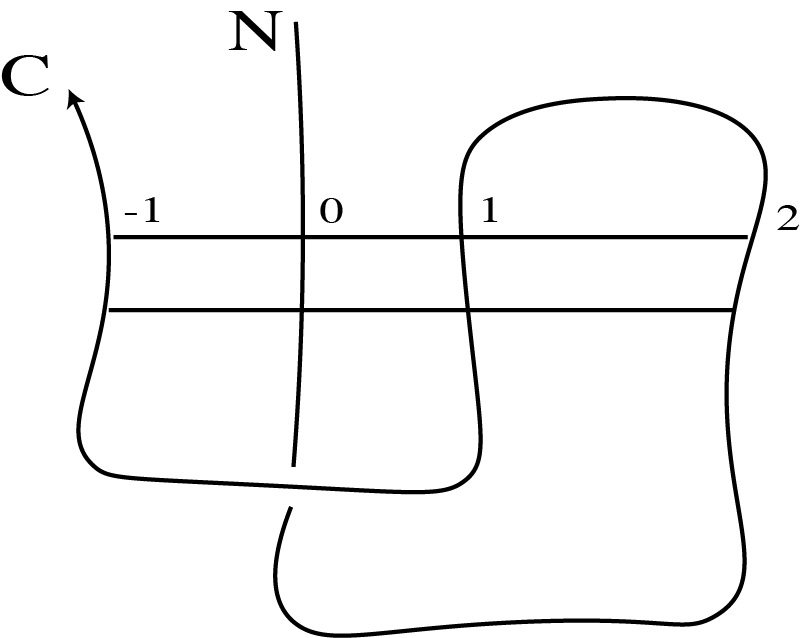}
  \caption{A $\beta$-pleated sheet in a projection, with Gauss code $N$ $\beta 1^+_0$ $\beta 1^-_2$ $\beta 1^+_1$ $\beta 1^-_{-1}$.}
  \label{fig:32}
\end{figure}

Finally, each crossing, $\alpha$-helix, bond, or $\beta$-pleated sheet is denoted with a sequential numbering based on its first appearance in the protein. Figure \ref{fig:33} gives an example of a Gauss code for a complete protein projection. 

\begin{figure}[htbp]
  \includegraphics[width=.5\linewidth]{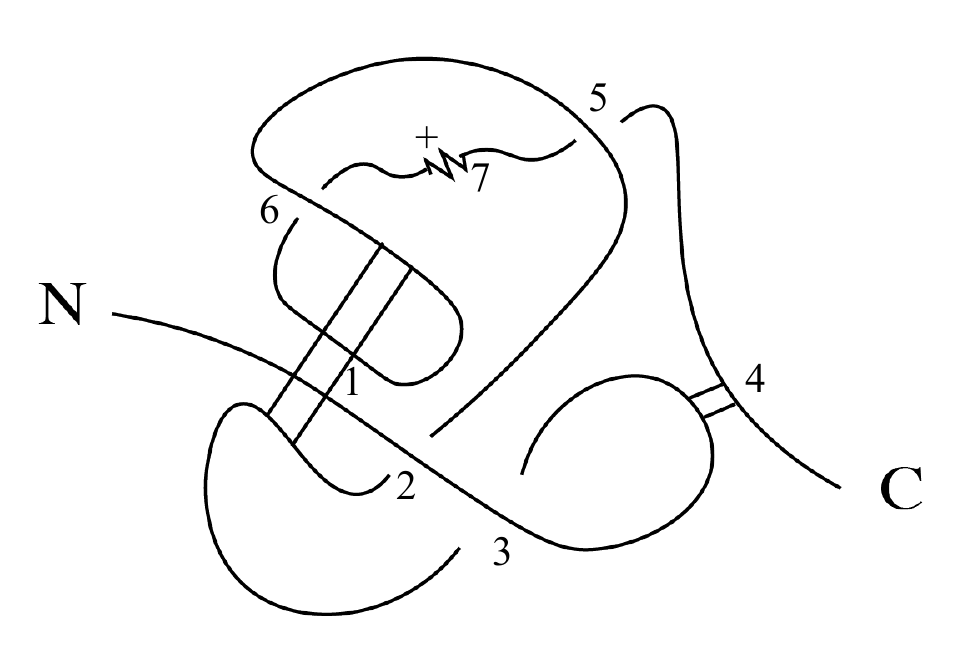}
  \caption{Example of a projection of a protein conformation with Gauss code given by  $N$ $\beta 1^+_0$ $O2^+$ $O3^-$ $B4^+$ $U3^-$ $\beta 1^+_1$ $U2^+$ $O5^-$ $O6^+$ $\beta 1^+_{-2}$ $\beta 1^-_{-1}$ $U6^+$ $\alpha 7^+$ $U5^-$ $B4^-$ $C$.}
  \label{fig:33}
\end{figure}

When we apply Reidemeister moves to a given projection, the move will be reflected in certain changes to the Gauss code. For instance, a Type I move inserts $On^\pm Un^\pm$ or $Un^\pm On^\pm$ into the Gauss code at the relevant point. Similar operations hold for all the Reidemeister moves. 

For example, for the protein projection appearing in Figure \ref{fig:33}, we could apply a Type VI Reidemeister move to slide the $\alpha$-helix out of the bigon (region in the projection plane bounded by two edges of the projection) bounded by crossings 5 and 6, and then remove the bigon by a Type II Reidemeister move to result in the Gauss code $N$ $\beta1^+_0$ $O2^+$ $O3^-$ $B4^+$ $U3^-$ $\beta 1^+_1$ $U2^+$ $\beta 1^+_{-2}$ $\beta 1^-_{-1}$ $\alpha 5^+$  $B4^-$ $C$.

But caution should be exercised. The corresponding operations on the Gauss codes do not always correspond to actual Reidemeister moves. For example, to do a Type II Reidemeister move, we must have two strands of the projection that are on the same complementary face of the projection. This is not visible from the Gauss code. 

\section{Quandles}

A \textit{knot invariant} is a map $I : \mathcal{K} \rightarrow S$ from all knot diagrams $\mathcal{K}$ to some set $S$ such that for any two projections $K_1$ and $K_2$ of the same knot type, $I(K_1) = I(K_2)$. The set $S$ can be a collection of integers, groups, polynomials or other mathematical objects. An invariant is said to be a \textit{complete invariant} if the converse is true, which is to say $I(K_1) = I(K_2)$ implies $K_1$ and $K_2$ represent the same knot type.


A quandle is an algebraic object that was introduced as an invariant for knots in 1982 independently in \cite{Joyce82} and \cite{Matveev84}. It has turned out to be a particularly effective means to distinguish knots.  For more details on quandles, see for example \cite{EN}. 


\begin{definition}A \textit{quandle} is a set $X$ with an operation $\rhd:X \times X \rightarrow X$ such that the following three conditions are satisfied. 
\begin{eqnarray*}
(1) && \text{For all } x \in X, x \rhd x = x. \\
(2) && \text{There exists an inverse function $\rhd^{-1}$ such that for all } x, y \in X,\\
&& (x \rhd y) \rhd ^{-1} y = x = (x \rhd ^{-1} y) \rhd y. \\
(3) && \text{For all } x, y, z \in X, (x \rhd y) \rhd z = (x \rhd z) \rhd (y \rhd z).
\end{eqnarray*}
\end{definition}


A slightly more restrictive algebraic structure than a quandle is a kei, also called an involutory quandle.

\begin{definition}A \textit{kei}, or  \textit{involutory quandle} is a set $X$ and operation $\rhd: X \times X \rightarrow X$ that satisfy the following three conditions. 

\begin{eqnarray*}
(1) && \text{For all } x \in X, x \rhd x = x. \\
(2) && \text{For all } x, y \in X, (x \rhd y) \rhd y = x = (x \rhd y) \rhd y. \\
(3) && \text{For all } x, y, z \in X, (x \rhd y) \rhd z = (x \rhd z) \rhd (y \rhd z).
\end{eqnarray*}
\end{definition}


Note that the only difference is that for an involutory quandle, $\rhd$ is equivalent to $\rhd^{-1}$. 


Depending on the situation, as we will discuss, one or the other of these algebraic structures may be the more appropriate to apply. 


A \textit{coloring} of an oriented knot projection by a quandle is an assignment of a value from $X$ to each arc, where an \textit{arc} is defined as part of a strand in a projection that both starts and ends at an under crossing, but going over zero or as many crossings as we like. We require that the labels assigned to the arcs be related through the quandle operation as in Figure \ref{crossingrelations}.

\begin{figure}[htbp]
  \includegraphics[width=.8\linewidth]{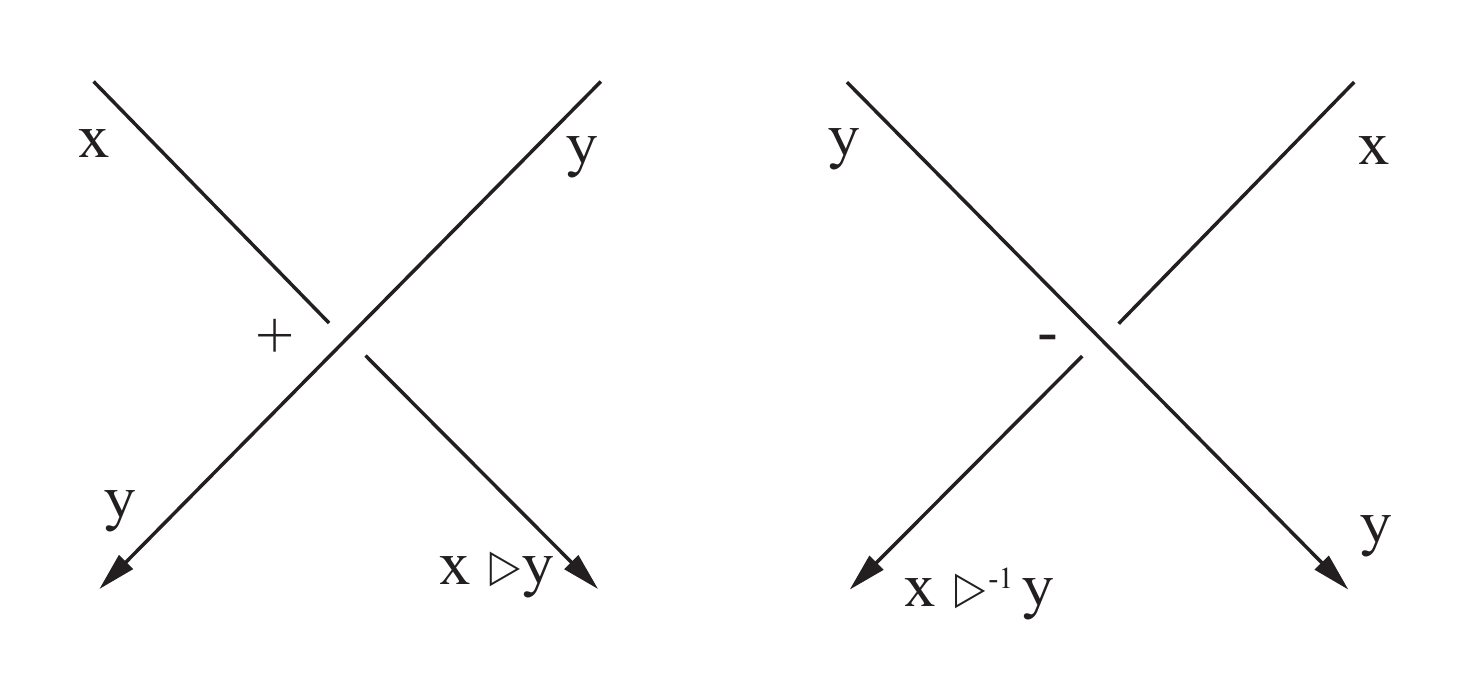}
  \caption{Quandle conditions that must be satisfied at a crossing.}
  \label{crossingrelations}
\end{figure}


The relevance of quandles to knots becomes apparent when we consider how the Reidemeister moves affect our labelled diagram as in Figure \ref{Reidemeisterquandle}. We see that the quandle axioms satisfied by the labels ensure that the quandle coloring is still valid after the Reidemeister moves. This means that the validity of the quandle coloring does not depend on the particular projection. It just depends on the knot type.  


\begin{figure}[htbp]
  \includegraphics[width=.9\linewidth]{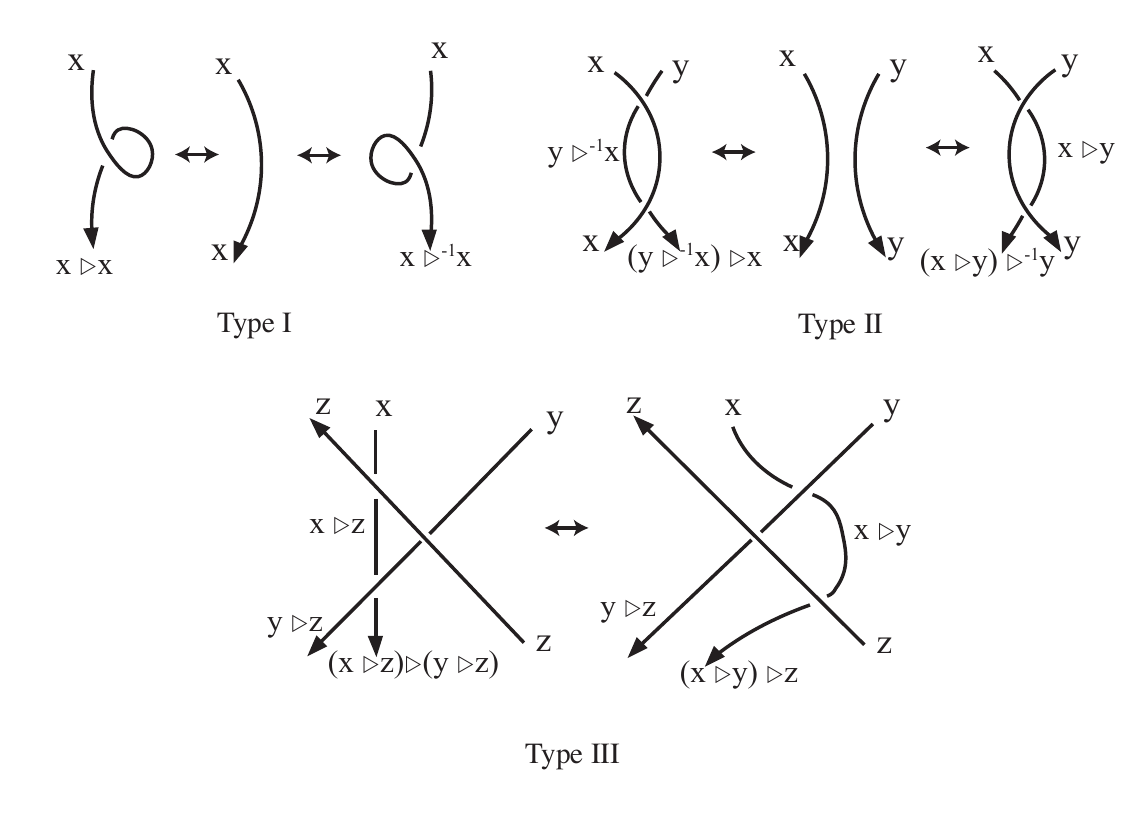}
  \caption{The quandle relations guarantee the Reidemeister moves respect the labels.}
  \label{Reidemeisterquandle}
\end{figure}

Thus, given a particular quandle, we can generate an invariant for knots by seeing how many distinct colorings by that quandle a particular knot has.  Two knots with different numbers of colorings by that quandle must then be distinct knots.

We can also drop the orientation on the knots, in which case $\rhd$ and $\rhd^{-1}$ become identical, the arrows disappear in Figure  \ref{Reidemeisterquandle},  and we color with involutory quandles instead. This simplifies things as we only have one operation to consider instead of two.

\section{Quandles and Singularities}

In order to allow for singularities in knots, the authors of \cite{CIEHN2016} introduced the {\it singquandle}. We first consider the singquandle for an unoriented knot, which will be an involutory quandle that satisfies additional conditions.

An arc in a singular knot projection is a strand that begins and ends at either an under-crossing or a singularity. Given an involutory quandle coloring of the arcs of a projection, we require the labels to satisfy conditions at the singular crossings as in Figure \ref{singularity1}, where $R_1(x,y)$ and $R_2(x,y)$ are maps from $X \times X$ to $X$ yet to be specified. 

\begin{figure}[htbp]
  \includegraphics[width=.3\linewidth]{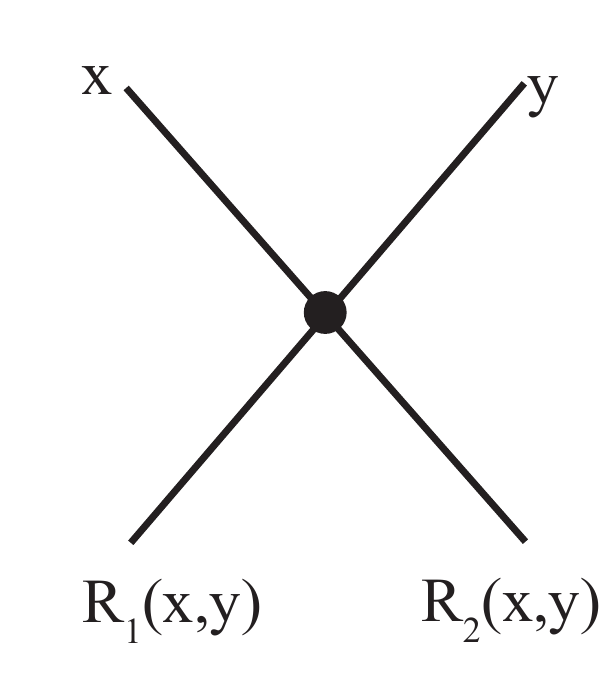}
  \caption{Labels at a singularity.}
  \label{singularity1}
\end{figure}

Since the diagram in Figure \ref{singularity1} can be  rotated by 90 degrees, 180 degrees and 270 degrees clockwise and the relation between the top pair of labels and the bottom pair of labels must be maintained, we immediately obtain certain relations that must be satisfied:

\begin{align}
x & = R_2(R_2(x,y), R_1(x,y)) & (\rm{rotate}\hspace{.05in} 180^o) \\
y & = R_1(R_2(x,y), R_1(x,y)) & (\rm{rotate}\hspace{.05in} 180^o)\\
x & = R_1(y, R_2(x,y)) & (\rm{rotate}\hspace{.05in} 270^o)\\
R_1(x,y) & = R_2(y , R_2(x,y)) & (\rm{rotate}\hspace{.05in} 270^o)\\
y & = R_2(R_1(x,y), x) & (\rm{rotate}\hspace{.1in} 90^o)\\
R_2(x,y) &  = R_1(R_1(x,y), x)) & (\rm{rotate}\hspace{.1in} 90^o)
\end{align}

In \cite{CIEHN2016}, the authors show that in the presence of singularities, the only additional Reidemeister moves necessary  are those coming from sliding a separate vertical strand on the left to the right behind or in front of a singularity,  or flipping a singularity as in Figure \ref{Reidemeistersing}.

\begin{figure}[htbp]
  \includegraphics[width=.5\linewidth]{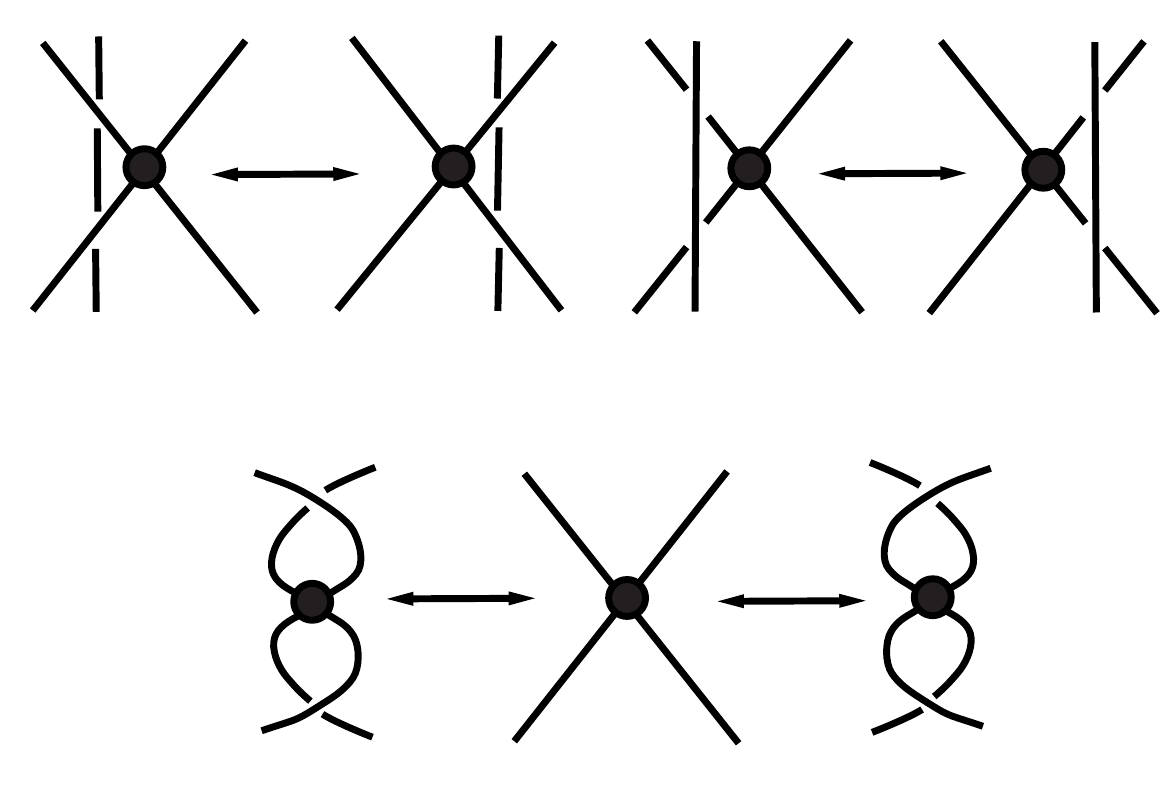}
  \caption{Reidemeister moves for a singularity.}
  \label{Reidemeistersing}
\end{figure}

These moves generate the additional relations:

\begin{align}
(y \rhd z) \rhd R_2(x,z) & = (y \rhd x) \rhd R_1(x,z) \\
R_1(x,y) & = R_2(y \rhd x, x) \\
R_2(x,y) & = R_1(y \rhd x, x) \rhd R_2(y \rhd x, x) \\
R_1(x \rhd y, z) \rhd y & = R_1(x, z \rhd y) \\
R_2(x \rhd y, z) &  = R_2(x, z \rhd y) \rhd y
\end{align}

\begin{definition} A {\it singquandle} is an involutory quandle, with a choice of $R_1(x,y)$ and $R_2(x,y)$ that satisfy all of the additional relations (1)-(11). 
\end{definition}

\bigskip

But the singularities we wish to consider for proteins are not of this type. In our case, we have bonds across two parallel strands, as in Figure \ref{singularity}.

\begin{figure}[htbp]
  \includegraphics[width=.3\linewidth]{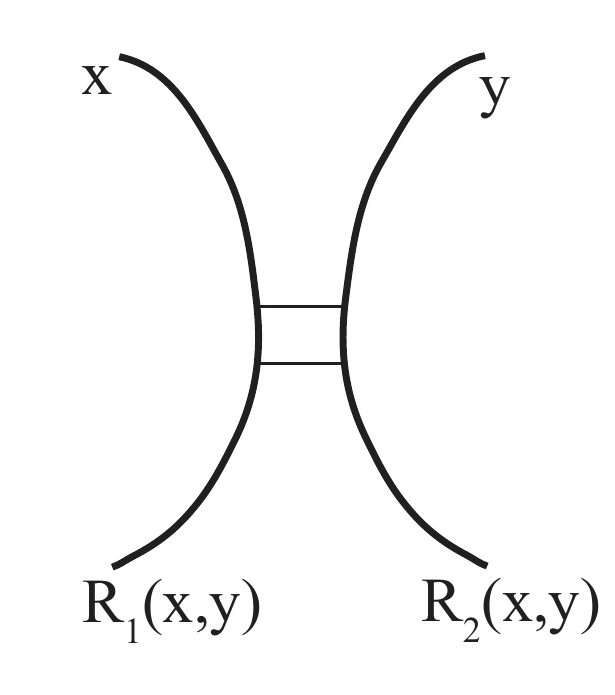}
  \caption{Labels at a bond.}
  \label{singularity}
\end{figure}

Such a bond does not have a four-fold rotational symmetry, but only a two-fold rotational symmetry. Thus, we have the following definition:

\begin{definition} An {\it involutory bondle} is an involutory quandle that satisfies the relations (1), (2), (7), (8), (9), (10) and (11).
\end{definition}

Although this choice allows us to incorporate bonds into our quandle, we do not yet have a way to represent $\beta$-pleated sheets. To deal with them, we replace a $\beta$-pleated sheet by a sequence of independent adjacent singular crossings as follows.

We have already assigned positive and negative integer values  to the strands in a $\beta$-pleated sheet from the subscripts of the Gauss code. Therefore, if we define the direction of the zero strand as downwards, we can define the relative heights of the individual singularities replacing a $\beta$-pleated sheet to be strictly increasing as the numbering increases, as shown in Figure \ref{fig:42}. The bonds appear as a set of stairs, either rising to the right or left, depending on which is the positive side of the labels on the $\beta$-pleated sheet. This transformation of a $\beta$-pleated sheet into adjacent singularities is called a {\it segmentation} of the $\beta$-pleated sheet.

\begin{figure}[htbp]
  \includegraphics[width=.6\linewidth]{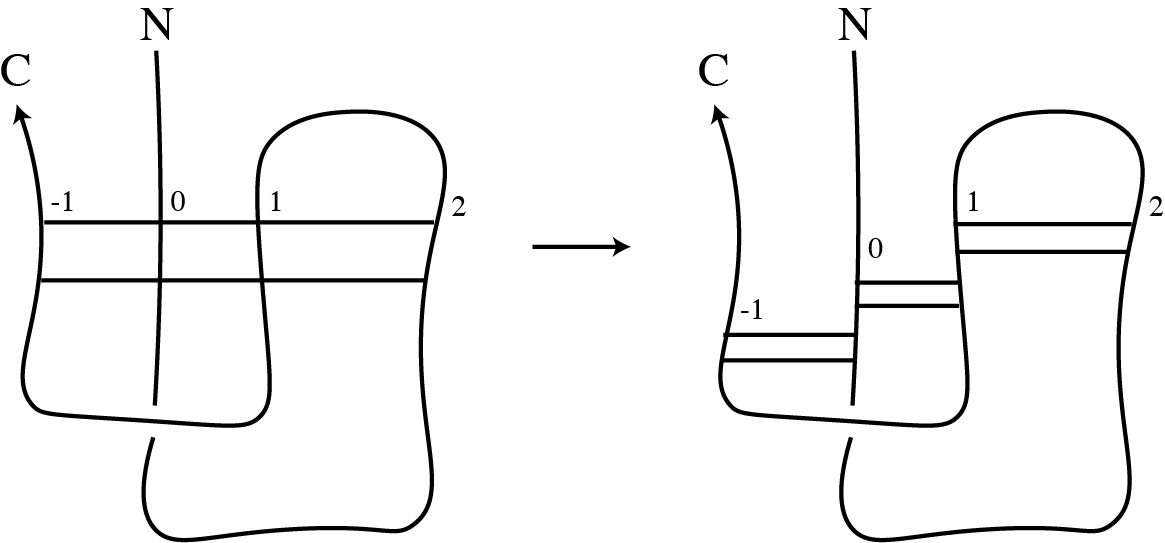}
  \caption{A $\beta$-pleated sheet before and after segmentation.}
  \label{fig:42}
\end{figure}

With this, we can still perform the Type IV and Type V moves on multi-singular crossing utilizing a sequence of Reidemeister moves on order two singularites, as shown in Figure \ref{fig:43}. Thus, no multi-singularity Reidemeister moves are needed. However, this choice for how to represent a $\beta$-pleated sheet does mean that we cannot distinguish between a protein with a $\beta$-pleated sheet and an identical one that has  the corresponding sequence of bonds in place of the $\beta$-pleated sheet. Bondle invariants will be equivalent for the diagrams in Figure \ref{fig:44}.

\begin{figure}[htbp]
  \includegraphics[width=.7\linewidth]{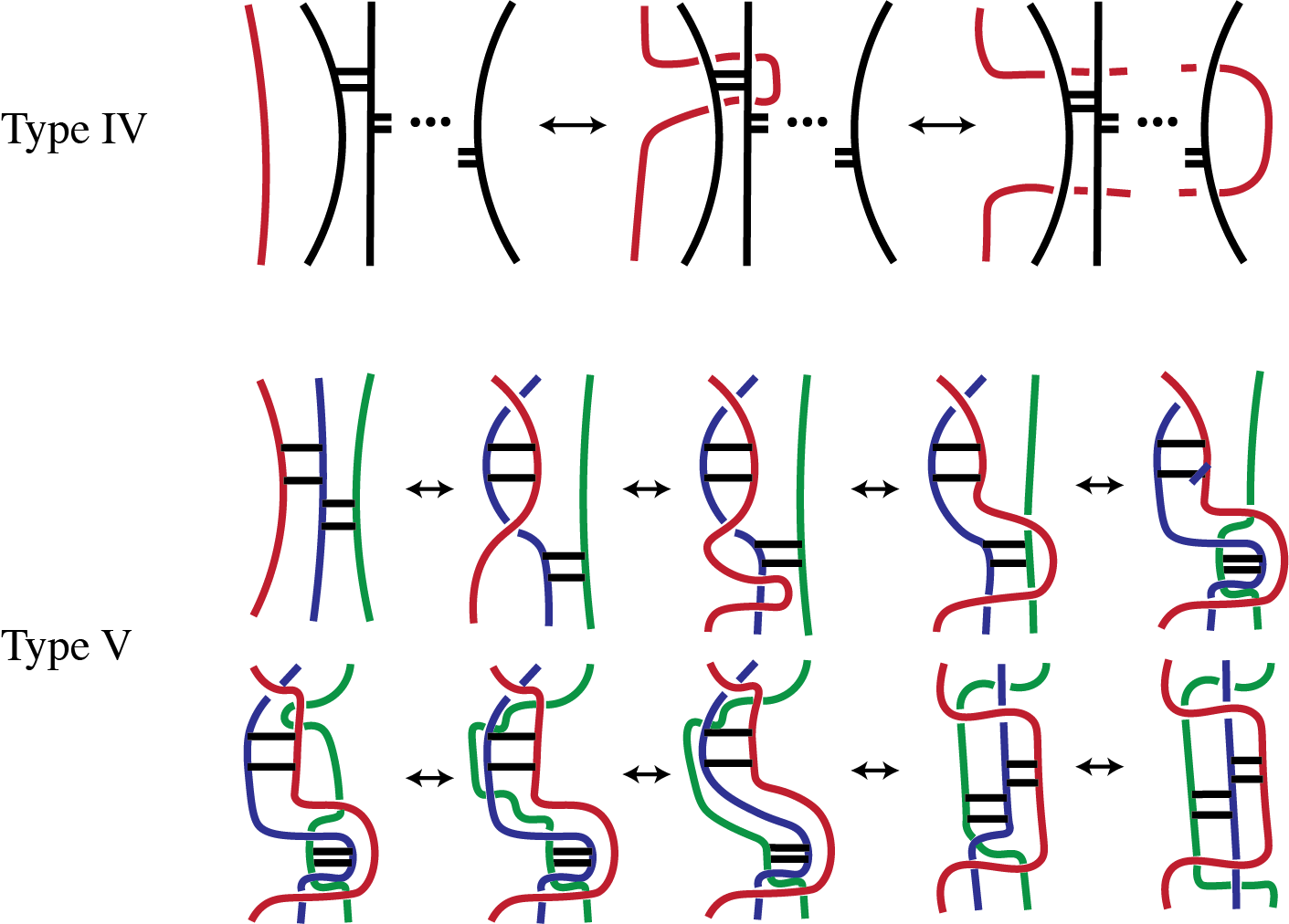}
  \caption{Reidemeister moves on a segmented $\beta$-pleated sheet.}
  \label{fig:43}
\end{figure}

\begin{figure}[htbp]
  \includegraphics[width=.8\linewidth]{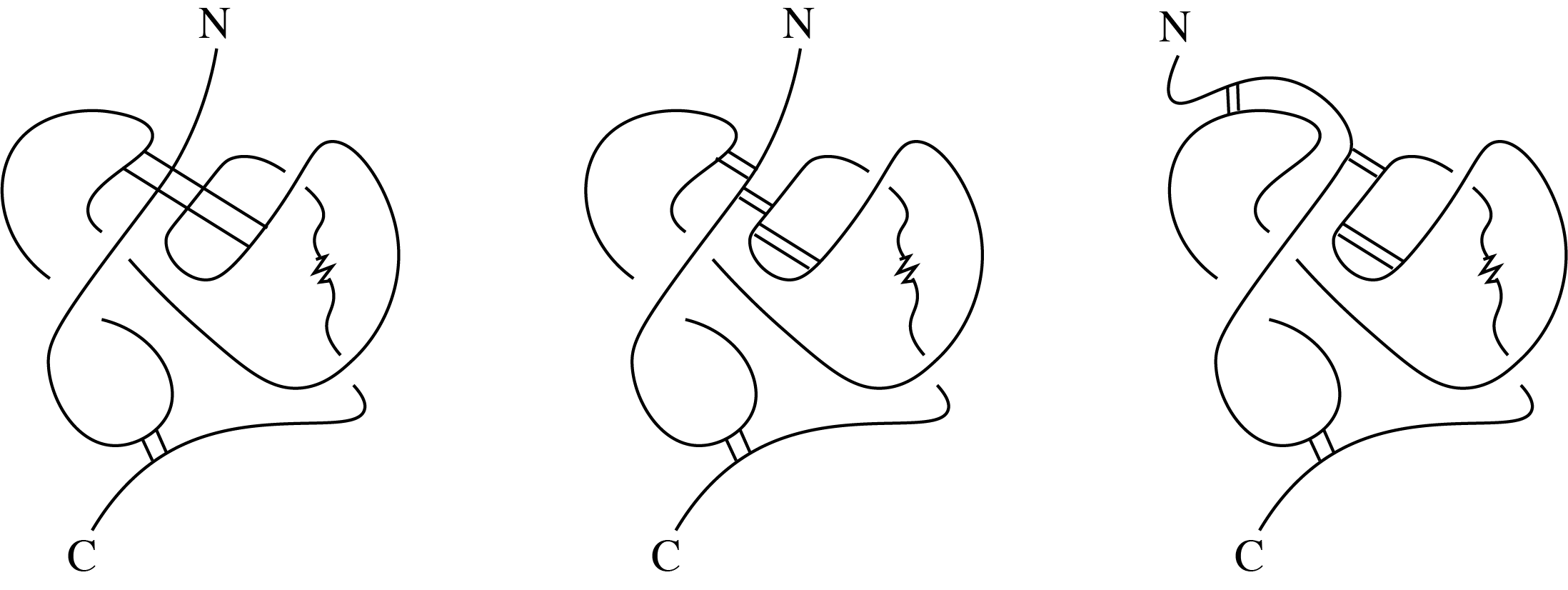}
  \caption{Three indistinguishable protein models.}
  \label{fig:44}
\end{figure}

The next issue we need to consider is the endpoints of the protein model. When considering proteins, we can view them as knot segments, with the ends free to move. Although in the physical realization of a protein,  ends are sometimes tucked inside the protein or are subject to constraints and are therefore not free to move, in our model, we allow them to slide past strands in any given projection. Even for a fixed rigid conformation, as we change our projection direction, the endpoints in the projections can slide past strands, eliminating or creating crossings. 
Therefore, we treat the end strands as insignificant until they reach the first bond. We think of the ends as only being relevant in defining the first and last bonds, and ignoring them otherwise, as in Figure \ref{fig:45}. 

\begin{figure}[htbp]
  \includegraphics[width=.5\linewidth]{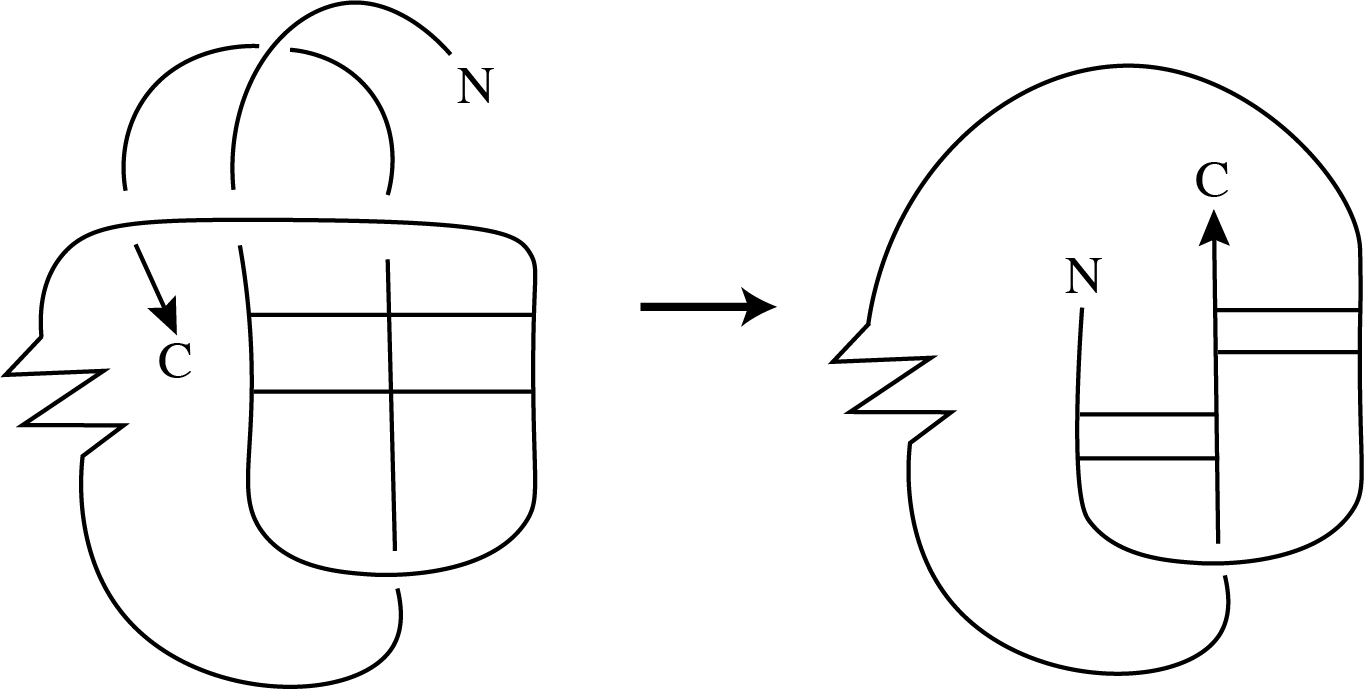}
  \caption{Reducing End Arcs}
  \label{fig:45}
\end{figure}

The final structure we need to consider is the $\alpha$-helix.  We view it as a sequence of $n$ kinks, where $n$ is the number of full rotations that the helix contains, all having either $+$ or $-$ crossings depending on whether it is a clockwise or counterclockwise $\alpha$-helix, as in Figure \ref{helix2}. These kinks are referred to as {\it residues}. Just as we are unable to distinguish a segmented $\beta$-pleated sheet from a sequence of adjacent bonds, we cannot distinguish an $\alpha$-helix from a sequence of kinks.

\begin{figure}[htbp]
  \includegraphics[width=.85\linewidth]{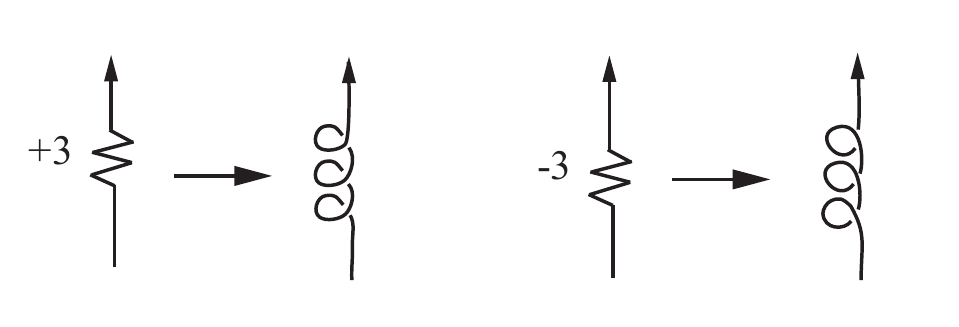}
  \caption{Replacing an $\alpha$-helix with a sequence of kinks.}
  \label{helix2}
\end{figure}

When coloring a protein with a quandle, the $\alpha$-helix becomes invisible because the Reidemeister Type I move in Figure \ref{Reidemeisterquandle} allows for the removal of kinks. However, there is a generalization of a quandle called a rack that does not allow for the removal of kinks, and therefore does see the existence of an $\alpha$-helix. A rack is simply a set that satisfies the second and third axioms of a quandle but not the first. We will not pursue racks further here.

\bigskip

\section{The Oriented Bondle} 
Since proteins do have a natural orientation, we should also consider the oriented version of the bondle.  The oriented singquandle was defined in \cite{BEHY2017}. The authors showed that in addition to the four traditional Reidemeister moves on oriented diagrams that were shown in  \cite{Polyak10} to suffice for oriented links (appearing as the first four in Figure \ref{orientedsingular}), the 14 possible Reidemeister moves involving singularities for oriented links can be reduced to the three depicted in Figure \ref{orientedsingular}. Thus seven Reidemeister moves suffice for equivalency of singular diagrams.

\begin{figure}[htbp]
  \includegraphics[width=.65\linewidth]{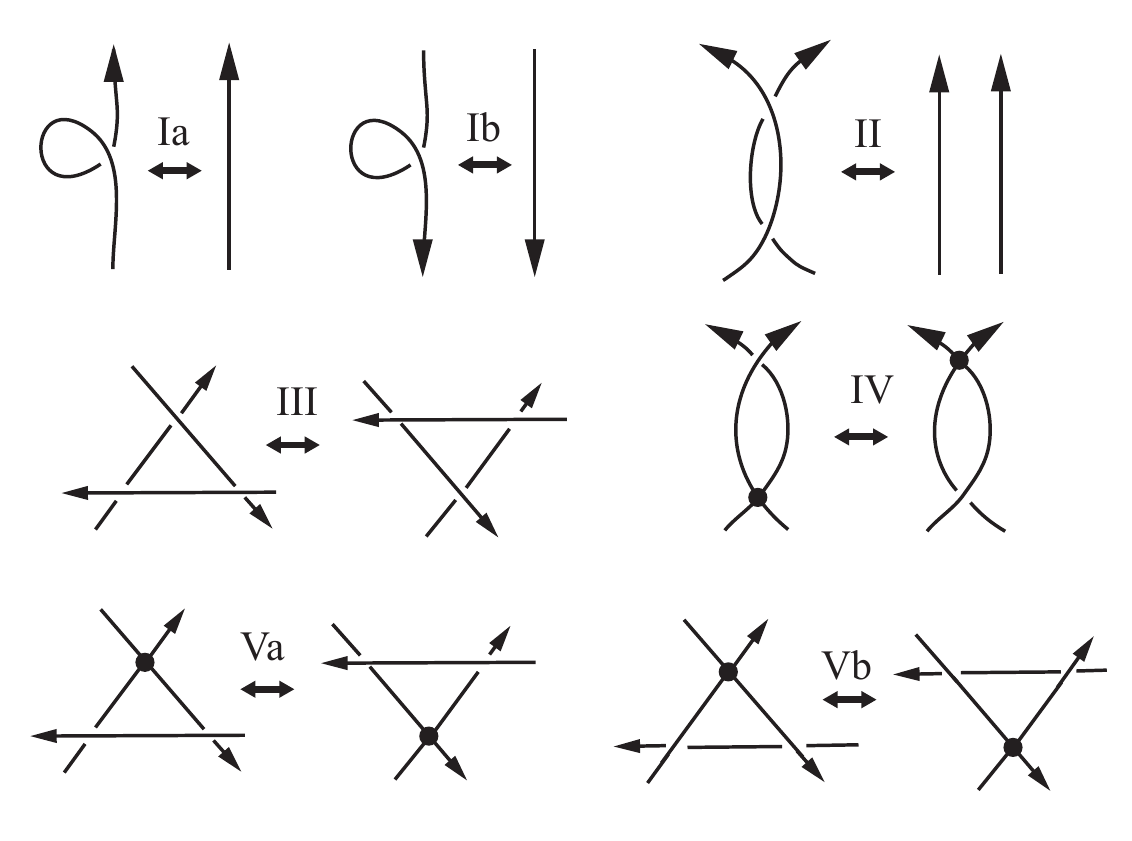}
  \caption{A generating set of Reidemeister moves for oriented singular knots.}
  \label{orientedsingular}
\end{figure}

Inserting our labels as in Figure \ref{singularity1} (but with both strands pointing downward) into these three possibilities, we obtain the a set of axioms to go with our three traditional quandle axioms coming from the non-singular moves.

\begin{definition}\label{ Def 1}

Let $(X, \tr)$ be a quandle.  Then if $R_1$ and $R_2$ are two maps from $X \times X$ to $X$ satisfying the following relations, we say that $(X, \tr)$ is an \it{oriented singquandle}.

	\begin{eqnarray}
	R_1(x {\tr}^{-1} y,z)\tr y&=&R_1(x,z\tr y)  \label{eq1}\\
	R_2(x {\tr}^{-1} y, z) & =&  R_2(x,z\tr y) {\tr}^{-1} y   \label{eq2}\\
	(y{\tr}^{-1} R_1(x,z))\tr x   &=& (y\tr R_2(x,z)) {\tr}^{-1} z  \label{eq3}\\
	R_2(x,y)&=&R_1(y,x\tr y)  \label{eq4}\\
	R_1(x,y)\tr R_2(x,y)&=&R_2(y,x\tr y)  \label{eq5}	
	\end{eqnarray}
\end{definition}


Note that for the oriented singquandle, there are no axioms coming from successive rotations by 90 degrees of Figure \ref{singularity1}.
The authors of \cite{BEHY2017} give the following two examples of singquandles.

\begin{example}\label{Example1}
		{\rm
	Let $n$ be a positive integer, and let $a$ be an invertible element in $\mathbb{Z}_n$ and $b$ any element in $\mathbb{Z}_n$.  Then the binary operations $x \tr y = ax+(1-a)y$, $x \tr^{-1} y = a^{-1}x+(1-a^{-1})y$, $R_1(x,y) = bx + (1-b)y$ and $R_2(x,y)= a(1-b)x + [b+ (1-b)(1 - a)]y $ make the triple $(\mathbb{Z}_n,\tr, R_1,R_2)$ satisfy the conditions to be an oriented singquandle.
}
\end{example}

	\begin{example}\label{Example2}
			{\rm
			Let $X=G$ be a non-abelian group with the binary operation $x \tr y=y^{-1}xy$.  Then, for $n \geq 1$, the following families of maps $R_1$ and $R_2$ make $(X, \tr, R_1, R_2)$ into an oriented singquandle:
			\begin{enumerate}
				\item
				$R_1(x,y)=x(xy^{-1})^{n}$
				and $R_2(x,y)=y(x^{-1}y)^n$,
				\item
				$R_1(x,y)=(xy^{-1})^nx$ and 
				$R_2(x,y)=(x^{-1}y)^ny,$

				\item
				$R_1(x,y)=x(yx^{-1})^{n+1}$ and 
				$R_2(x,y)=x(y^{-1}x)^n.$

			\end{enumerate}
}		
	\end{example}

\bigskip

In the case of proteins, we would like to consider bonds rather than singularities.  There are two distinct types of oriented bonds, one where the orientations on the two strands are parallel and one where they are anti-parallel, as in Figure \ref{singularity3}.

\begin{figure}[htbp]
  \includegraphics[width=.6\linewidth]{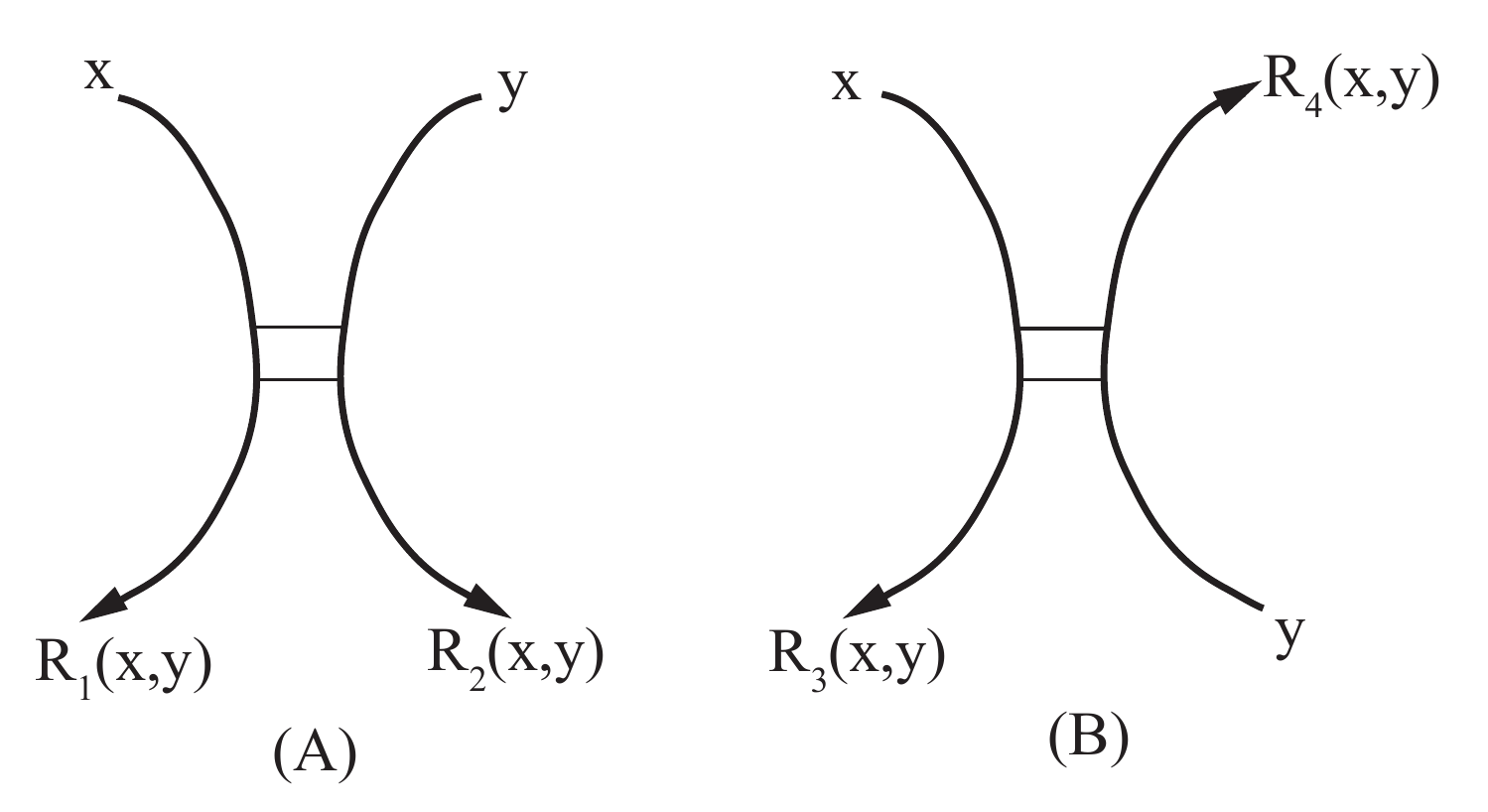}
  \caption{Labels at bonds with parallel and anti-parallel strands.}
  \label{singularity3}
\end{figure}

Each of the fourteen moves involving singularities from \cite{BEHY2017} yields two possibilities corresponding to whether the singularity is replaced with a vertical or horizontal bond. However, it is still true that for each of a vertical or horizontal bond, the fourteen moves reduce to three. So in addition to the four non-singular Reidemeister moves, we have six more moves to consider.

The first three correspond to bond diagram (A) in Figure \ref{singularity3}, and we inherit the same set of relations as for the singquandle, namely (12)-(16). 

Considering  bond diagram (B) from Figure \ref{singularity3}, we pick up two more functions $R_3(x,y)$ and $R_4(x,y)$. But note that rotation by 180 degrees switches the roles of $x$ and $y$ and the roles of $R_3$ and $R_4$. Thus, it is always the case that $R_4(x,y) = R_3(y,x)$. We will use this to eliminate $R_4(x,y)$ from all subsequent relations.

From Figure \ref{orientedrelations}, we obtain four additional relations.

\begin{figure}[htbp]
  \includegraphics[width=.75\linewidth]{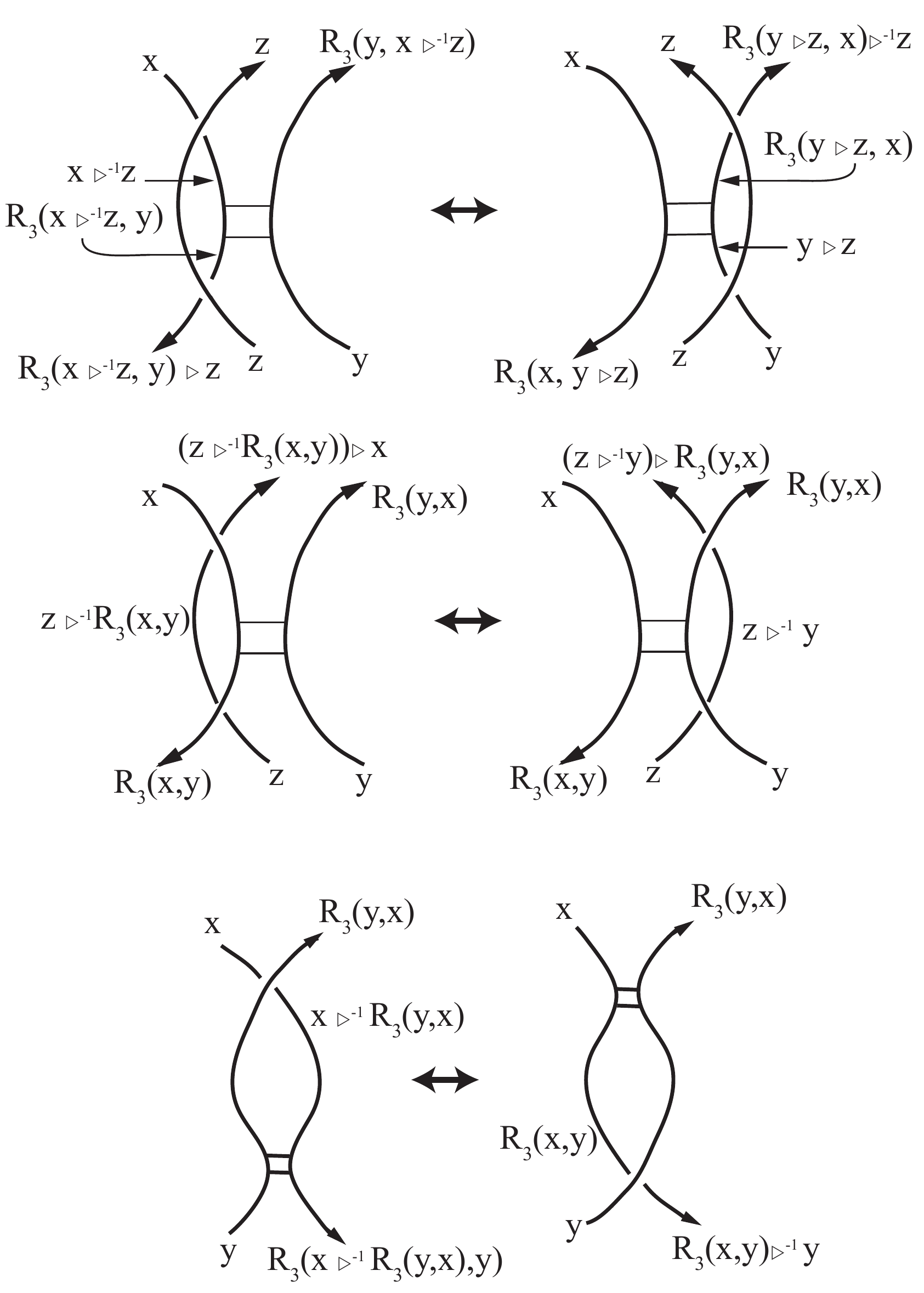}
  \caption{Relations from oriented bonds.}
  \label{orientedrelations}
\end{figure}

\begin{definition} \label{Def 2}An {\it oriented  bondle} is a quandle with operation $\rhd$ and choices for functions $R_1(x,y), R_2(x,y)$ and $R_3(x,y)$ such that they satisfy relations (12)-(16) and the additional relations: 
\begin{align}
R_3(y, x \rhd^{-1} z) & = R_3(y \rhd  z, x ) \rhd^{-1} z \label{eq6}\\
R_3(x, y \rhd z) & = R_3(x \rhd^{-1} z, y ) \rhd z  \label{eq7}\\
(z \rhd^{-1} R_3(x,y)) \rhd x & = (z \rhd^{-1} y) \rhd R_3(y,x)\label{eq8}\\
R_3(x,y) \rhd^{-1} y & = R_3( x \rhd^{-1} R_3(y,x), y).\label{eq9}
\end{align}
\end{definition}

	Note that the maps $R_3(x,y)=x$ and $R_3(x,y)=y$ do always satisfy the relations (\ref{eq6}),  (\ref{eq7}),  (\ref{eq8}) and  (\ref{eq9}) for any quandle $(X,\tr)$.  We call these {\it trivial solutions} as they do not recognize the existence of the bond.

Since we already have examples of the desired maps $R_1$ and $R_2$ for both Example~\ref{Example1} and Example~\ref{Example2}, we would like to find some solutions for the map $R_3$ satisfying relations (\ref{eq6}),  (\ref{eq7}),  (\ref{eq8}) and  (\ref{eq9}) . \\
	
	\begin{lemma}
		Let $n$ be a positive odd integer greater than or equal to 3 and let $a$ be an invertible element of $\mathbb{Z}_n$. Consider the quandle $(\mathbb{Z}_n, \tr)$ with $x \tr y= a x + (1-a)y$ and inverse operation $x \tr^{-1} y= a^{-1}x+ (1 - a^{-1})y$.  
		Let $m$ be an element in $\mathbb{Z}_n$ and let $R_3$ be given by $R_3(x,y) = mx+(1-m)y$. Then the map $R_3$ satisfies the equations (\ref{eq6}),  (\ref{eq7}),  (\ref{eq8}) and  (\ref{eq9}) if and only if $m(m-1)=0 \in \mathbb{Z}_n$.	
	\end{lemma}
	
	\begin{proof}
		Direct computations show that the map $R_3$ given by $R_3(x,y) = mx+(1-m)y$ satisfies the three equations  (\ref{eq6}),  (\ref{eq7}),  (\ref{eq8}).  Now substituting $R_3$ in equation (\ref{eq9})  and simplifying gives the condition $m(m-1)(x-y)=0$, for all $x, y \in \mathbb{Z}_n$, and thus yields the condition $m(m-1)=0 \in \mathbb{Z}_n$.
		
	\end{proof}
	We then have the following corollary
	\begin{corollary}\label{Cor}
		Let $n=pq$ where $p$ and $q$ are odd primes.  Assume further that $x \tr y= a x + (1-a)y$ and $x \tr^{-1} y= a^{-1}x+ (1 - a^{-1})y$,  for invertible element $a$ in $\mathbb{Z}_n$. For any fixed element $b$ in $\mathbb{Z}_n$, let $R_1(x,y) = bx + (1-b)y$ ,  $R_2(x,y)= a (1-b)x + [b+ (1-a)(1-b)]y $ and $R_3(x,y) = mx+(1-m)y$.  Then $( \mathbb{Z}_n, \tr, R_1, R_2, R_3)$ is an oriented bondle if and only if $p$ divides $m$ and $q$ divides $(m-1)$ or  $p$ divides $(m-1)$ and $q$ divides $m$. 
	\end{corollary}
	The following is a list of some $(n,m)$ satisfying Corollary~\ref{Cor}.
	\begin{enumerate}
		\item 
		If $n=15$ then $m=6$ or $m=10$.
		
		\item
			If $n=21$ then $m=7$ or $m=15$.
			
		\item
			If $n=33$ then $m=12$ or $m=22$.
		\item
			If $n=35$ then $m=15$ or $m=21$.\\

	\end{enumerate}





Now we consider the case when the quandle is a group $G$ with conjugation.  First recall that the commutator of two elements $x$ and $y$ in a group $G$ is given by $[x,y]:=xyx^{-1}y^{-1}$.  We have the following Lemma.

\begin{lemma}
	Let $X=G$ be a non-abelian group and let the quandle operation on $G$ be given by $x \tr y=y^{-1}xy$, so that $x\; {\tr}^{-1} y=yxy^{-1}.$  Assume that $R_3$ is given by $R_3(x,y) = x^py^q$, where $p$ and $q$ are integers, then
	
	\begin{enumerate}
		
		\item
		
		The map $R_3$ satisfies both equation (\ref{eq6}) and equation  (\ref{eq7}) for any integers $p$ and $q$.
		\item 
		If for all $x,y \in G$, $x^{p-1}y^q= x^{-q}y^{1-p}$ then $R_3$ satisfies equation (\ref{eq8}).
		
		\item
			Let $p$ be an integer.  If for all $x,y \in G$, the commutator $[x^p, y^{1-p}] =1,$ then $R_3$ satisfies equation (\ref{eq9}).
		
	\end{enumerate}
\end{lemma}

\begin{proof}
	Assume that $R_3$ has the form $R_3(x,y) = x^py^q$, then 
		\begin{enumerate}
			
			\item
			One can see that equation (\ref{eq6}) is satisfied for all integers $p$ and $q$ from the following. 
			\begin{eqnarray*}
		R_3(y, x\; {\tr}^{-1} z)&=&y^p(zxz^{-1})^q=y^pzx^qz^{-1}= zz^{-1}y^pzx^qz^{-1}\\
		&=&R_3(y \tr z,x)\;{\tr}^{-1} z.  
	\end{eqnarray*}
	Similarly, one has 
		\begin{eqnarray*}
	R_3(x, y \tr z)&=&x^pz^{-1} y^qz=z^{-1}zx^p z^{-1} y^qz= R_3(x\; {\tr}^{-1} z ,y) \tr z,
	\end{eqnarray*}
	showing that equation (\ref{eq7}) is satisfied also for all integers $p$ and $q$.\\
	
	\item
	Now we check equation (\ref{eq8}).  Assume that the equation $x^{p-1}y^q= x^{-q}y^{1-p}$ holds in $G$.  Now we compute both the left hand side (LHS)  and the right hand side (RHS) of equation (\ref{eq8}). 
	\begin{eqnarray*}
	LHS &=&(z {\tr}^{-1} R_3(x,y)) \tr x  = x^{-1}\; x^py^q\;z\; y^{-q}x^{-p}x \\
	&=& x^{p-1}y^q \; z \; y^{-q}x^{1-p}, \\
RHS &=& [R_3(y,x)]^{-1} yzy^{-1} R_3(y,x)=(y^px^q)^{-1}	yzy^{-1} y^px^q\\
&=& x^{-q}y^{1-p}\; z \; y^{p-1}x^q.
\end{eqnarray*}	
Since $x^{p-1}y^q= x^{-q}y^{1-p}$, then $LHS=RHS$ giving the result.
	
	\item
	We finish by checking equation (\ref{eq9}).  Here also we compute separately the LHS and the RHS.  We thus have 
	\begin{eqnarray*}
		LHS &=& R_3(x,y)\; {\tr}^{-1} y= yx^p y^q y^{-1}= yx^py^{q-1}\\		
		RHS &=& R_3(x\; {\tr}^{-1} R_3(y,x), y)=(y^px^qxx^{-q}y^{-p})^p y^q\\
		    &=& y^px^py^{-p}y^q= y^px^py^{q-p}.
\end{eqnarray*}

Now since the commutator $[x^p, y^{1-p}]=1$, then we have $x^p y^{1-p}=y^{1-p} x^p$.  Multiplying this equation by $y^p$ from the left and by $y^{q-1} $ from the right gives the equation $y^p x^p y^{q-p}=y x^p y^{q-1}$, thus
we have $RHS=LHS$ giving equation (\ref{eq9}).
\end{enumerate} 
\end{proof}

In order to give a more explicit example of a non-abelian group with a map $R_3$ satisfying equations (\ref{eq8}) and (\ref{eq9}), we use the group of symmetries of a square.

\begin{definition}
Given a square with vertices labeled by $1,2,3$ and $4$, let $G$ be the set of all rigid motions of the square that send vertices to vertices.  Under composition, this set forms a non-abelian group called the {\it dihedral group of order 8} and denoted $D_4$. Precisely, $D_4=\{1, r, r^2, r^3, s, sr, sr^2, sr^3\}$, where the permutation $r=(1\;2\;3\;4)$ is the clockwise rotation of $90$ degrees and $s$ is the reflection $s=(1\;2)(3\;4)$.
\end{definition}

Recall that in $D_4$, the elements $r$ and $s$ satisfy the relations $r^4=1=s^2$ and $srs=r^{-1}$.  By iterating this last identity, we obtain $sr^i s=r^{-i}$, for $0 \leq i \leq 3$. 
The square of any element of $D_4$ is either the identity element $1$ or $r^2$.  Then we see that $r^2$ commutes with any other element of $D_4$, since $sr^i\; r^2= (sr^{i+2}s)\;s= r^{-i-2}s=r^{4-i-2}s=r^{2-i}s=r^2(sr^is)s=r^2\; sr^i$.

\begin{corollary}
In the dihedral group $D_4$, the maps $R_3(x,y)=x^2y^{-1}$ and  $R_3(x,y)=x^{-1}y^2$ both satisfy equations (\ref{eq8}) and (\ref{eq9}). 
\end{corollary}

We thus obtain the following family of bondles.

\begin{example}\label{ExampleBond}
			{\rm
		  Let $X=D_4$ be the quandle with operation $x \tr y=y^{-1}xy$, 
			then the following families of maps $R_1, R_2$ and $R_3$ make $(X, \tr, R_1, R_2, R_3)$ into a bondle:
			\begin{enumerate}
				\item
				$R_1(x,y)=x(xy^{-1})^{n}$,  $R_2(x,y)=y(x^{-1}y)^n$ and $R_3(x,y)=x^2 y^{-1}$ 
				\item
				$R_1(x,y)=(xy^{-1})^nx$, 
				$R_2(x,y)=(x^{-1}y)^ny$ and $R_3(x,y)=x^2 y^{-1}$

				\item
				$R_1(x,y)=x(yx^{-1})^{n+1}$, 
				$R_2(x,y)=x(y^{-1}x)^n$ and $R_3(x,y)=x^2 y^{-1}$. 
			\end{enumerate}
}		
	\end{example}
	Note that this example still holds if we change $R_3(x,y)=x^2 y^{-1}$  to $R_3(x,y)=x^{-1}y^2$. 

\bigskip

\section{Examples}
Given two projections of proteins and a choice of bondle, we can count the number of distinct colorings of each projection by that bondle, and if those numbers are distinct, we know the two proteins are not topologically equivalent. This provides an opportunity for the categorization of proteins into distinct topological types. 
In the following we give two examples demonstrating the use of \emph{oriented bondles} to topologically distinguish proteins with bonds.  
\begin{example}
\rm{
            In this example, we use the oriented bondle $( \mathbb{Z}_{15}, \tr, R_1, R_2, R_3)$ from Corollary~\ref{Cor} with $a = 8$.  Since $8 \times 2$ is congruent to $1$ modulo $15$ then $a^{-1} = 2$.  We set $b=2$ and then use this oriented bondle to distinguish the topological type of the following two two proteins $P_1$ and $P_2$.
            
\begin{figure}[htbp]

  \includegraphics[width=.75\linewidth]{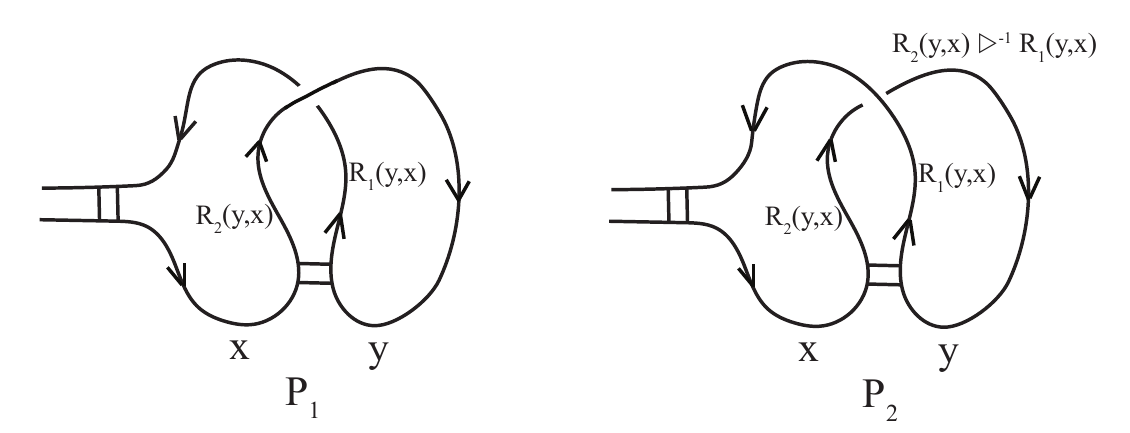}
  \caption{Distinguishing $P_1$ from $P_2$.}
  \label{HopfBond}
\end{figure}

            Precisely, $x \tr y= 8(x+y),$ $x \tr^{-1}y=2x-y$, $R_1(x,y)=2x-y$ and $R_2(x,y)=7x-6y$. Note that we do not need to define $R_3(x,y)$ because there are no anti-parallel bonds in the diagrams.
            
            A coloring of $P_1$ gives the following equation
            
              $$             R_2(y,x)=y $$

            which simplifies to
$$6(y-x)=0          $$
           So if 5 divides $y-x$, we obtain a nontrivial coloring. Thus, the total number of colorings, including the trivial colorings, is 45.
           
           On the other hand,  a coloring of $P_2$ gives the following equation
            $$R_2(y,x) \tr^{-1} R_1(y,x) = y$$
            which simplifies to
           $$11 (y-x) = 0.$$  
               
  Since $11$ is invertible in $\mathbb{Z}_{15}$, we see that $x=y$, implying that $P_2$ has \emph{only} trivial colorings, of which there are 15.  Thus $P_1$ and $P_2$ are distinct. 
  }         
            
\end{example}

\begin{example}
       \rm{ In this example, we include anti-parallel bonds. We utilize the bondle $\mathbb{Z}_{15}$ with $a = 7$, $a^{-1} =13= -2$ (modulo $15$), $b=8$ and $m=6$. Thus,  $x \tr y = 7x - 6y$ and $x \tr^{-1}y=-2x+3y$,  $R_1(x,y)=8x-7y$,  $R_2(x,y) = -4x + 5y$ and $R_3(x,y) = 6x-5y$.
           \begin{figure}[htbp]
           
  \includegraphics[width=1\linewidth]{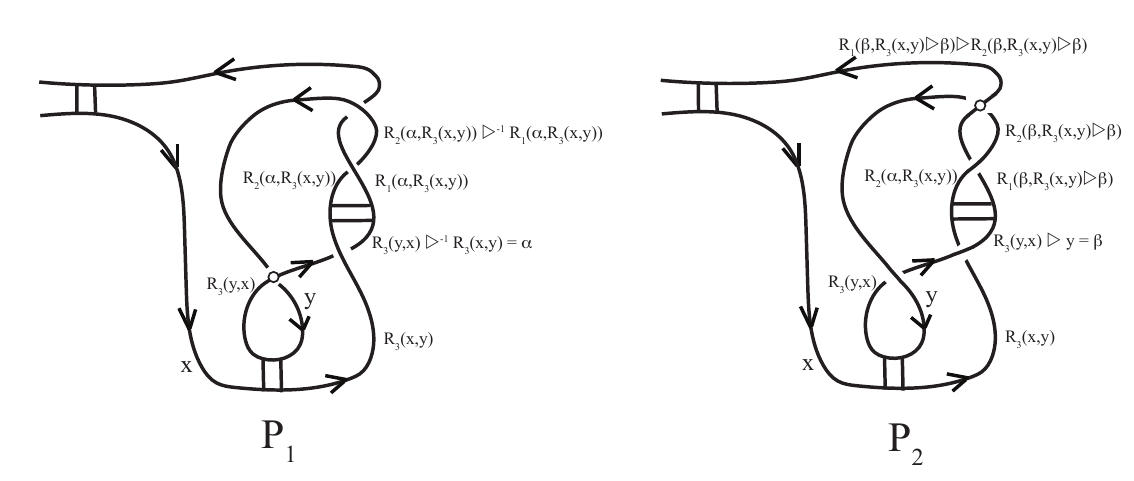}
  \caption{Distinguishing $P_1$ from $P_2$.}
  \label{distingusihprotein2}
\end{figure}

At the crossing with a hollow dot in $P_1$, we obtain the relation:

$$ y = [R_2(\alpha, R_3(x,y) \tr^{-1} R_1(\alpha, R_3(x,y)] \tr^{-1}R_3(y,x)$$

\medskip
  This yields $0 = 5(x-y)$, implying that there are nontrivial colorings corresponding to when 3 divides $x-y$. So we obtain a total of 75 colorings.
  
  But at the crossing with a hollow dot in $P_2$, we obtain the relation:
  
  \[ y = R_2(\beta, R_3(x,y) \tr \beta) \tr [R_1(\beta, R_3(x,y) \tr \beta) \tr R_2(\beta, R_3(x,y) \tr \beta)]]
           \]
          
This yields $0 = 7(y-x)$, and as 7 is invertible, we only obtain the  15 trivial colorings corresponding to  $y = x$. Thus, the two proteins must be topologically distinct.
 }
\end{example}

\section{Conclusion}

When intra-chain interactions are included for linear molecules, a rich knot theory is possible. Utilizing some of the standard tools of knot theory extended to this new paradigm, including generalized Reidemeister moves and Gauss codes, it is possible to catalog the various knotted structures that result. To that end, the extension of quandles to bonded linear segments, called bondles, allows for the differentiation of the topological structures that can appear. This approach could be mechanized, allowing for computers to search for the parameters for the appropriate bondle to distinguish between the topological types of two proteins, for instance. There are many avenues for further research in these directions. 

\section*{Acknowledgement} Thanks to Jack Roche for suggesting the term ``bondle''.

\medskip

\bibliographystyle{amsalpha}
\bibliography{Thesis.bib}

\end{document}